\documentclass[a4paper,11pt]{amsart}

\usepackage[margin=1in]{geometry}
\usepackage{lscape}
\usepackage{amsthm,float}
\usepackage{latexsym,bezier,caption,amsbsy,psfrag,color,enumerate,amsfonts,amsmath,amscd,amssymb,comment}
\usepackage{epsfig}
\usepackage{rotating}
\usepackage{graphics}
\usepackage{graphicx}
\usepackage{stmaryrd}
\usepackage{extarrows}
\usepackage{amsmath}
\usepackage{cite}
\usepackage{dsfont}
\usepackage{mathtools}
\usepackage{turnstile}
\usepackage{tikz}
\usepackage{url}
\usepackage{hyperref}
\usepackage{tikz-cd}

\DeclareSymbolFont{rsfscript}{OMS}{rsfs}{m}{n}
\DeclareSymbolFontAlphabet{\mathrsfs}{rsfscript}

\DeclareMathOperator{\Aut}{Aut}

\DeclareMathOperator{\Sch}{Sch}
\DeclareMathOperator{\Geo}{Geo}

\DeclareMathOperator{\Cay}{Cay}
\DeclareMathOperator{\Star}{Star}
\DeclareSymbolFont{rsfscript}{OMS}{rsfs}{m}{n}
\newtheorem{theorem}{Theorem}
\newtheorem{prop}{Proposition}
\newtheorem{defn}{Definition}
\newtheorem{lemma}{Lemma}

\newtheorem{rem}{Remark}
\newtheorem*{theorem*}{Theorem}

\newcommand{\la}{\langle}
\newcommand{\ra}{\rangle}

\newcommand{\wt}{\widetilde}

\newcommand{\oo}{\overline}

\def\vlongrightarrow{\relbar\joinrel\longrightarrow}
\def\vvlongrightarrow{\relbar\joinrel\vlongrightarrow}
\def\vvvlongrightarrow{\relbar\joinrel\vvlongrightarrow}

\def\longmapright#1{\smash{\mathop{\vlongrightarrow}\limits^{#1}}}
\def\vlongmapright#1{\smash{\mathop{\vvlongrightarrow}\limits^{#1}}}
\def\vvlongmapright#1{\smash{\mathop{\vvvlongrightarrow}\limits^{#1}}}

\newcommand{\mapright}[1]{\smash{\stackrel{\text{\tiny{$#1$}}}{\vlongrightarrow}}}

\begin{document}
\title{Generalizations of the Muller-Schupp theorem and tree-like inverse graphs}
\author{Emanuele Rodaro}
\address{Emanuele Rodaro, Politecnico di Milano - Piazza Leonardo da Vinci, 32 20133 Milano, Italia}
\email{emanuele.rodaro@polimi.it}



\begin{abstract}
We extend the characterization of context-free groups of Muller and Schupp in two ways. We first show that for a quasi-transitive inverse graph $\Gamma$, being quasi-isometric to a tree, or context-free (finitely many end-cones types), or having the automorphism group $\Aut(\Gamma)$ that is virtually free, are all equivalent conditions. Furthermore, we add to the previous equivalences a group theoretic analog to the representation theorem of Chomsky-Schützenberger that is fundamental in solving a weaker version of a conjecture of T. Brough which also extends Muller and Schupp' result to the class of groups that are virtually finitely generated subgroups of direct product of free groups. We show that such groups are precisely those whose word problem is the intersection of a finite number of languages accepted by quasi-transitive, tree-like inverse graphs. 
\end{abstract}

\maketitle

\noindent {\footnotesize{\bf Mathematics Subject Classification (2020)}: 20F65, 05C75, 20F10, 68Q70, 20M18, 57M10.}
\smallskip

\noindent {\footnotesize{\bf Keywords}: group actions on inverse graphs, quasi-tree, word-problem, virtually free groups, context-free graphs and languages, subgroups of the direct product of free groups.}


\section{Introduction}
The word problem of a finitely generated group $G$ presented by $\la A|\mathcal{R}\ra$ consists of the set $WP(G; A)$ o all words in $(A\cup A^{-1})^{*}$ that represent the identity of $G$. One of the most interesting problems in this setting is the algebraic characterization of a group $G$ in terms of the language theoretic properties of  $WP(G; A)$. For example, Anisimov proved that a group is finite if and only if its word problem is a regular language \cite{asimo}. A famous result in this area is the classification of the groups with context-free word problems (context-free groups) which states that a finitely generated group $G$ has a context-free word problem if and only if $G$ is virtually free; a result that follows from the work of Muller and Schupp \cite{muahu} together with the work of Dunwoody \cite{dun} on the accessibility of finitely presented groups. Moreover, geometrically context-free groups are characterized by Cayley graphs which are quasi-isometric to trees \cite[Theorem 19 pp. 133]{harpe}. In this paper, we extend these results to the class of inverse graphs, i.e., graphs in the sense of Serre labeled on an involutive alphabet $A\cup A^{-1}$ such that for every $a\in A\cup A^{-1}$ and vertex $v$ there is at most one edge $e$ starting at $v$ labeled by $a$, and consequently, the inverse edge $e^{-1}$ has $v$ as its terminal vertex and it is labeled by $a^{-1}$. These graphs often appear in several areas of mathematics, especially geometric group and semigroup theory via Cayley and Schreier graphs in the area of group theory, and Sch\"{u}tzenberger graphs that are the analog of Cayley graphs in the area of inverse semigroup theory. An inverse graph $\Gamma$ is called quasi-transitive if it has a finite number of orbits under the action of its automorphism group $\Aut(\Gamma)$. 
One of the main results of the paper, Theorem~\ref{theo: main theorem}, stands at the intersection of graph theory, geometric group theory, and theoretical computer science:
\begin{theorem*}
Let $\Gamma$ be an infinite quasi-transitive inverse graph. T.F.A.E.
\begin{enumerate}
\item $\Gamma$ is context-free (it has finitely many end-cones types);
\item $\Gamma$ is quasi-isometric to a tree (it is tree-like);
\item $\Aut(\Gamma)$ is a finitely generated virtually free group;
\item there is a cover $\psi: \Gamma\to \Lambda$ of inverse graphs with $\Lambda$ finite, and a surjective morphism
 $$\oo{\eta}:\pi_1(\Lambda, \psi(x_0))\to \mathbb{F}_X$$ from the fundamental group $\pi_1(\Lambda, \psi(x_0))$ onto a free group $\mathbb{F}_X$ of finite rank, with $\ker(\oo{\eta})$ that is equal to the set of reduced words in the language $L(\Gamma, x_0)$ of the labels of circuits based at the root $x_0$.

\end{enumerate}
\end{theorem*}
In the same spirit of the previous theorem, in \cite{LiWo} the authors characterize quasi-transitive deterministic $A$-digraphs whose languages of the self-avoiding walks are context-free with respect to the size of the ends. 
The generalization of the notion of context-freeness from groups to a generic inverse graph is not immediate at first, since there might be several ways to give such a definition. Since for a group the word problem corresponds to the set of labels of the walks starting and ending in a vertex of the Cayley graph, we can transfer this definition textually to an inverse graph.
In Section~\ref{sec: context-free graphs} we prove that this definition is equivalent to several others, including the more graph theoretic one of context-free graph given by  Muller and Schupp in \cite{muahu2} as a labeled graph with finitely many labeled end-cones types. We also characterize such graphs as the configuration graphs of inverse pushdown automata, a variant of reversible pushdown automata introduced in \cite{kutrib}. 
\\
To prove the implication $(1)\Rightarrow (2)$ of our characterization, in Section~\ref{sec: geometrical considerations} we show that any context-free inverse graph is quasi-isometric to a tree. As a consequence, the converse of \cite[Theorem 4.13]{Pedro Nora} holds, providing a solution to \cite[Question 6.2]{Pedro Nora} and the following geometric characterization
\begin{theorem*}
Let $M$ be a finitely presented inverse monoid. Then, a Sch\"{u}tzenberger graph of some element of $M$ is context-free if and only if it is quasi-isometric to a tree.
\end{theorem*}
We emphasize that also this result generalizes the analogous one for groups since a group is a special case of inverse monoid and the (unique) Sch\"{u}tzenberger graph of a group coincides with its Cayley graph. While the equivalence $(2)\Leftrightarrow (3)$ can be obtained directly from standard facts in literature, to close all the equivalences we pass from the technical condition $(4)$ which is a group theoretic representation theorem \`a la Chomsky–Schützenberger. This condition is a key fact that it is used in the last section to shed some light on a conjecture regarding poly-context-free groups. As a natural extension of context-free groups, and thus to virtually free groups, Brough in \cite{Tara} has introduced the concept of poly-context-free groups. A group is called poly-context-free if its word problem is the intersection of a finite number of context-free languages. Since the intersection of context-free languages is not in general context-free, a natural problem is the classification of such groups. In \cite{Tara} the author has conjectured that the class of finitely generated poly-context-free groups coincides with the class of (finitely generated) groups which are virtually a finitely generated subgroup of a direct product of free groups. Using condition (4) in Section~\ref{sec: weaker tara's conjecture} we show the other main result of the paper which also offers another generalization of the theorem of Muller and Schupp. 
\begin{theorem*}
A group $G$ has the word problem in the class of languages that are intersections $\bigcap_{i=1}^kL(\Gamma_i, x_i)$ of finitely many languages accepted by quasi-transitive tree-like inverse graphs $\Gamma_i$, if and only if it is virtually a (finitely generated) subgroup of a direct product of free groups. 
\end{theorem*}


\section{labeled digraphs, involutive and inverse graphs}\label{sec:preliminaires}
Throughout the paper, we assume that $A = \{a_1\ldots  a_n\}$ is a finite set of letters that we call alphabet. The free semigroup on $A$ is denoted by $A^+$, and the free monoid on $A$ with empty word $1$ is denoted by $A^*$. Any subset of $A^*$ is called a language (over $A$). We denote by $\mathbb{F}_A$ the free group with basis $A$, that is, the group with presentation $\la A\mid \emptyset\ra$. We denote by $A^{-1}=\{a^{-1}:a\in A\}$ the set of formal inverses of the alphabet $A$, and by $\wt{A}=A\cup A^{-1}$ the involutive closure of $A$. We equip $\wt{A}^*$ with the usual involution $\circ^{-1}:\wt{A}^+\to \wt{A}^+$ defined by $(u^{-1})^{-1}=u$, $(uv)^{-1}=v^{-1}u^{-1}$ for all $u,v\in \wt{A}^+$. A word $w\in \wt{A}^*$ is called \emph{(freely) reduced} if it does not contain any factor of the form $uu^{-1}$ for some $u\in \wt{A}^*$. It is a well-known fact that given a word $w\in \wt{A}^*$ by successively removing pairs of consecutive inverse letters after a finite number of steps we end up with a unique reduced word $\oo{w}$, called the \emph{free reduction of $w$}. Each element of $\mathbb{F}_A$ may be identified with the equivalence class $[u]=\{w\in \wt{A}^*:\oo{w}=u\}$ for some unique reduced word $u$ which uniquely identifies such class. Thus, the free group $\mathbb{F}_A$ (with basis $A$) can then be thought as the set of reduced words in $\wt{A}^*$ with the operation defined by $\oo{u}\cdot \oo{v}=\oo{uv}$. With a slight abuse of notation we will denote the identity of $\mathbb{F}_A$ by $1$, while for a generic group $G$, we will use the notation $\mathds{1}_G$. For a generic subset $L\subseteq \wt{A}^*$, by $\oo{L}=\{\oo{u}:u\in L\}$ we denote the set of all the reduced words that may be obtained from $L$. 
\\
An $A$-digraph (or $A$-graph, or just graph) is a tuple $\Gamma=(V,E,\iota, \tau, \lambda)$ where $V=V(\Gamma)$ is the set of vertices, $E=E(\Gamma)$ the set of oriented edges, $A$ the set of labels, $\lambda:E\to A$ is the labeling function, and $\iota:E\to V$, $\tau:E\to V$ are the initial and final (incident) functions, where for all $e\in E$ we interpret $\iota(e)$, $\tau(e)$ as the initial and terminal vertices of the edge $e$, respectively. Throughout the paper, we mostly adopt a visual approach to such graphs, and we will graphically represent a generic edge $e\in E$ with $\omega(e)=v_1$, $\tau(e)=v_2$ and $\lambda(e)=a$ as $v_1\mapright{a}v_2$. Although there are no edges labeled by the empty word, in some arguments sometimes will be useful to consider empty edges of the form $v_1\mapright{1}v_1$. 
We say that $\Gamma$ is \emph{deterministic} (\emph{complete}) if for any vertex $v$ and $a\in A$ there is at most (at least) one edge $v\mapright{a}v'$ for some $v'\in V(\Gamma)$; and it is \emph{co-deterministic} if for any vertex $v'$ and $a\in A$ there is at most one edge $v\mapright{a}v'$ incident to $v'$. We reserve the name \emph{simple graphs} for a pair $(V,E)$ where $V$ is a set of vertices and the set of unoriented edges $E$ is a $2$-subset of $V$; a generic edge $e$ is of the form $\{v_1,v_2\}$ where $v_1,v_2\in V$ are the endpoints. The \emph{associated simple graph} to the $A$-direct graph $\Gamma$ is the simple graph with the same set of vertices $V(\Gamma)$ of $\Gamma$ and where two vertices $v_1, v_2$ are connected by an edge if there is an edge $v_1\mapright{a} v_2$ or $v_2\mapright{b} v_1$ for some $a,b\in A$. 
\\
An $A$-digraph $\Gamma'=(V',E',\omega, \tau,\lambda)$ with $V'\subseteq V$, $E'\subseteq E$ is called a \emph{subgraph} of $\Gamma$, and we write $\Gamma'\subseteq \Gamma$. We say that $\Gamma'$ is \emph{induced} if for any $e\in E$ with $\iota(e)\in V'$ or $\tau(e)\in V'$ we have $e\in E'$. For any subgraph $\Gamma'$ of $\Gamma$, we denote by $\la \Gamma'\ra$ the intersection of all the induced subgraphs containing $V(\Gamma')$. The difference of $\Gamma$ with $\Gamma'$ is the subgraph $\Gamma\setminus \Gamma'=\la V(\Gamma)\setminus V(\Gamma')\ra$. A \emph{walk} of length $k$ is a subgraph consisting of a collection of successive edges $p=e_1\ldots e_k$ with the property $\iota(e_{i+1})=\tau(e_{i})$ for $i=1,\ldots, k-1$. The label of $p$ is the word $\lambda(p)=\lambda(e_1)\ldots \lambda(e_k)\in A^+$ and if we want to emphasize the label we will sometimes depict such walk as 
\[
p=x_1\mapright{a_1}x_2\mapright{a_2}\ldots x_k\mapright{a_k}x_{k+1}
\]
with $e_i=x_i\mapright{a_i} x_{i+1}$. The vertices $x_1,x_{k+1}$ are the initial and terminal vertices of $p$, respectively, and if we do not want to specify them we denote them as $\iota(p),\tau(p)$. For any $x_i, x_j$ with $i<j$, we use $p(x_i,x_j)$ to denote  the subwalk of $p$ with endpoints $x_i, x_j$. If $\Gamma$ is deterministic given the initial vertex $x_1$ and a word $w\in A^*$ there exists at most one path $p$ labeled by $w$ and initial vertex $x_1$, which we sometimes denote as $x_1\mapright{w}y$. In case the walk $p$ has the property $\omega(p)=\tau(p)$, $p$ is called a \emph{circuit}. 
When we fix a vertex $x_0$ in $\Gamma$, the pair $(\Gamma, x_0)$ is called a \emph{rooted graph}, or a (pointed) $A$-automaton. In this case, we see $(\Gamma, x_0)$ as a language acceptor by using the distinguished vertex $x_0$ as an initial and final state. In this case, the language accepted by $(\Gamma,x_0)$ is given by the set $L(\Gamma,x_0)$ formed by the labels of all circuits with initial vertex $x_0$. Note that $L(\Gamma,x_0)$ together with the usual operation of concatenation, is also a submonoid of the free monoid $A^*$. More generally, an $A$-automaton has an initial state and a set $F\subseteq V(\Gamma)$ of final states, the language $L(\Gamma, x_0,F)$ is the set formed by the labels of the walks $p$ with $\iota(p)=x_0$, $\tau(p)\in F$. In case $\Gamma$ is finite, $L(\Gamma, x_0,F)$ belongs to the class of regular languages \cite{hop}.
\\
The class of $A$-digraphs forms a category where a morphism $\varphi:\Gamma_1\to \Gamma_2$ from the graph $\Gamma_1=(V_1,E_1,\iota_1, \tau_1, \lambda_1)$ to the graph $\Gamma_2=(V_2,E_2,\iota_2, \tau_2, \lambda_2)$ is a pair of maps $\varphi_e:E_1\to E_2$, $\varphi_v:V_1\to V_2$ that commute with the initial and final maps and with the labeling function:
\[
\forall g\in E_1:\;\; \tau_2(\varphi_e(g))=\varphi_v(\tau_1(g)),\,\,  \iota_2(\varphi_e(g))=\varphi_v(\iota_1(g)),\, \,\lambda_2(\varphi_e(g))=\lambda_1(g)
\]
Given a subgraph $\Lambda$ of $\Gamma_1$ we denote by $\psi(\Lambda)$ the subgraph of $\Gamma_2$ with set of vertices $\varphi_v(V(\Lambda))$ and set of edges $\varphi_e(E(\Lambda))$. An isomorphism (monomorphism) is a morphism $\varphi$ such that $\varphi_v$, $\varphi_e$ are both bijective (injective).
\\
From now on we will mainly consider digraphs on alphabets with an involution which is the main object of this paper. 
\begin{defn}[involutive and inverse graphs]
An \emph{involutive} $\wt{A}$-digraph is a directed graph $\Gamma$ on the set of labels $\wt{A}$ with an involution $\circ ^{-1}:E\to E$ such that
\begin{itemize}
\item  for any $e\in E$, $e^{-1}$ is the opposite edge with $\iota(e^{-1})=\tau(e)$, $\tau(e^{-1})=\iota(e)$;
\item the involution is compatible with the labeling function: $\lambda(e^{-1})=\lambda(e)^{-1}$;
\item no edge is the inverse of itself $e\neq e^{-1}$. 
\end{itemize}
$\Gamma$ is called \emph{inverse} if, in addition, it is deterministic and connected, i.e., any two vertices are connected by a walk in $\Gamma$. 
\end{defn}
If $\Gamma$ is an $A$-digraph, then by adding for each edge $e=x_1\mapright{a}x_2$ in $\Gamma$ a new edge $e^{-1}=x_2\mapright{a^{-1}}x_1$ with label on the set of formal inverse $A^{-1}$, we obtain a new involutive graph on $\wt{A}$ that we denote by $\Gamma^-$. It is obvious that if $\Gamma$ is deterministic and co-deterministic, then $\Gamma^-$ is deterministic.  
\\
The notions that we have previously considered for an $A$-digraph, transfer directly to involutive and inverse graphs by adding the compatibility with the involution on the edges. Thus, in the category of involutive (inverse) graphs a morphism $\psi:\Gamma_1\to \Gamma_2$ of involutive (inverse) graphs is a morphism of $\wt{A}$-digraphs with the property that $\psi(e^{-1})=\psi(e)^{-1}$. In particular, considering the inclusion map, the subgraph of an involutive (inverse) graph is also involutive (inverse). In our context, a \emph{tree} is a connected graph in which every circuit $p=e_1\ldots e_k$ contains a subwalk of the form $ee^{-1}$ for some edge $e$ of $p$. A walk $p$ not containing any subwalk of the form $ee^{-1}$ is called \emph{reduced}. So in a tree there are no reduced walks, except the empty one. Note that the associated simple graph of an involutive graph that is a tree $T$, is a simple graph that is also a tree as a simple graph, i.e., every two vertices are connected by a unique simple path. 
\\
Given a connected involutive graph $\Gamma$, we can equip $\Gamma$ with a structure of a metric space by considering the usual distance given by
\[
d(u,v)=\min\{n: \mbox{ such that }p=e_1\ldots e_n\mbox{ is a walk with }\iota(p)=u, \tau(p)=v\}
\]
A walk $p$ is a geodesic, if $p$ has length exactly $d(\iota(p),\tau(p))$. The (induced) subgraph $D_n(p)$ consisting of vertices whose length from a fixed vertex $p$ is at most $n$ is called the \emph{disk of center $p$ and radius $n$}. For a connected subgraph $Y\subseteq \Gamma$ the diameter of $Y$ is given by $\delta_{\Gamma}(Y)=\max\{d(y_1,y_2): y_1,y_2\in V(Y)\}$.
\\
The following lemma is a direct consequence of the determinism of inverse graphs. 
\begin{lemma}\label{lem: there are reduced walks}
Let $\Gamma$ be an inverse $\wt{A}$-digraph. For any vertex $v\in V(\Gamma)$ and word $u\in\wt{A}^*$ there at most one walk $v\mapright{u}y$ labeled by $u$. Moreover, if there is a walk $x\mapright{u}y$, then there is also the (reduced) walk $x\mapright{\oo{u}}y$.
\end{lemma}
To fix the notation we now recall some basic definitions of groups acting on graphs. We assume the reader familiar with groups acting on trees, the Bass-Serre theory, and the fundamental group of a graph of groups, see for instance \cite{Serre, DiDun, Bogo}. We say that a group $G$ acts on an involutive $A$-digraph $\Gamma$ on the left if the left actions of $G$ on the sets $V(\Gamma)$ and $E(\Gamma)$ are defined so that $g\cdot\iota(e) =\iota(g\cdot e)$, $g\cdot \tau(e) = \tau(g\cdot e)$, $\lambda(g\cdot e)=\lambda(e)$ and $g\cdot e^{-1} = (g\cdot e)^{-1}$ for all $g\in G$, $e\in E(\Gamma)$. By the definition and the fact that $A\cap A^{-1}=\emptyset$ we deduce that $G$ acts on $\Gamma$ without inversion of edges ($g\cdot e\neq e^{-1}$ for all $e\in E(\Gamma)$ and $g\in G$). For $x\in V(\Gamma)\cup E(\Gamma)$ we denote by $G_x=\{g\in G:g\cdot x=x\}$ the stabilizer of $x$ under the action of $G$, and we denote by $G\cdot x=\{g\cdot x: g\in G\}$ the orbit of $x$ under the action of $G$. Put $G\setminus V(\Gamma)=\{G\cdot v: v\in V(\Gamma)\}$, $G\setminus E(\Gamma)= \{G\cdot e: e\in E(\Gamma)\}$. By the quotient graph $G\setminus \Gamma$ we mean the involutive graph $(G\setminus V(\Gamma), G\setminus E(\Gamma), \oo{\iota}, \oo{\tau}, \lambda)$ where $\oo{\iota}(G\cdot e)=G\cdot \iota(e)$, $\oo{\tau}(G\cdot e)=G\cdot \tau(e)$, $\lambda(G\cdot e)=\lambda(e)$.
\\
Henceforth, we will mostly consider inverse graphs which is the main subject of the paper. This category of digraphs appears in many areas of mathematics and computer science. For instance, given a group $G$ with generating set $A$, the Cayley graph $\Cay(G;A)$ is the inverse graph with vertex set $G$ and where a generic edge has the form $g\mapright{a}ga$ with $a\in\wt{A}$. A natural generalization is the notion of Schreier graphs. Given a subgroup $H$ of $G$ we consider the inverse graph $\Sch(H,A)$ with vertex set the right cosets $Hg$ of $H$, and edges of the form $Hg\mapright{a}H(ga)$, $a\in \wt{A}$. Via Stallings construction \cite{Stall}, one can associate to a subgroup $H$ of the free group $\mathbb{F}_A$ a geometric object, called the Stallings automaton $\mathcal{S}(H)$ which provides useful information about the subgroup $H$. One way to see the Stallings automaton of $H$ is by considering the core of the rooted Schreier graph $(\Sch(H), H)$, i.e., the (induced) rooted inverse subgraph obtained by taking all the vertices that belong to a reduced circuit $H\mapright{u} H$ with $u$ reduced. Roughly speaking, we are removing from $(\Sch(H), H)$ all the ``hanging trees'' not containing the root $H$. In this way the language $\oo{L(\mathcal{S}(H))}=\{\oo{u}: u\in L(\mathcal{S}(H))\}$ is formed by all the reduced words representing the elements of $H$, so it may be identified with $H$, see for instance \cite[Chapter 23]{Handbook}. Another area where inverse graphs play an important role is in combinatorial inverse semigroup theory. In this context, Sch\"{u}tzenberger automata play a key role, like Cayley graphs for geometric group theory. For more details, we refer the reader to \cite{Steph}.  



\subsection{Basic facts about morphisms of inverse graphs}
The condition of determinism and connectedness of an involutive graph has some consequences when we consider morphisms in the category of inverse graphs. For instance, a morphism $\psi:\Gamma\to\Lambda$ is completely determined by the vertex map $\psi_v$. In particular, $\psi$ is a monomorphism if for any $y\in V(\Lambda)$, the preimage $\psi_v^{-1}(y)$ contains at most one element and checking if two morphisms are equal is just a matter of verifying if they agree on a vertex.
\begin{lemma}\label{lem: uniqueness of automorphism}
Let $\varphi_1, \varphi_2$ be two morphisms from the inverse graph $\Gamma_1$ to the inverse graph $\Gamma_2$, then $\varphi_1=\varphi_2$ if and only if $\varphi_1(x)=\varphi_2(x)$ for some $x\in V(\Gamma)$.
\end{lemma}
\begin{proof}
Any walk $x\mapright{u}y$ is mapped by $\varphi_1, \varphi_2$ into the two walks 
\[
\varphi_1(x)\mapright{u}\varphi_1(y)\mbox{ and }\varphi_2(x)\mapright{u}\varphi_2(y) 
\]
and so by the determinism of $\Gamma_2$ and condition $\varphi_1(x)=\varphi_2(x)$ we deduce $\varphi_1(y)=\varphi_2(y)$. Since $\Gamma_1$ is connected, we have that $\varphi_1, \varphi_2$ agree on $V(\Gamma_1)$, hence by the determinism of $\Gamma_2$, they also agree on $E(\Gamma_1)$.
\end{proof}
The previous result is a particular case of a more general one since morphisms of inverse graphs are immersions of involutive graphs, the result follows from \cite[Proposition 2.1]{Gr-Meak}. 
Another interesting feature of inverse graphs, and in particularly rooted inverse graphs, is the relationship between accepted languages and morphisms. This connection is contained in the following proposition belonging to folklore. 
\begin{prop}\label{prop: inclusion implies homomorphism}
Let $(\Gamma,x_0), (\Lambda, y_0)$ be two rooted inverse graphs, then $L(\Gamma,x_0)\subseteq L (\Lambda, y_0)$ if and only if there is a morphism $\psi: \Gamma\rightarrow  \Lambda$ with $\psi(x_0)=y_0$. Moreover, $L(\Gamma,x_0)= L (\Lambda, y_0)$ if and only if there is an isomorphism $\psi: \Gamma\rightarrow  \Lambda$ 
\end{prop}
\begin{proof}
Suppose we have the inclusion $L(\Gamma,x_0)\subseteq L (\Lambda, y_0)$. Take a vertex $v\in V(\Gamma)$, since $\Gamma$ is connected there is a walk $x_0\mapright{u}v$ in $\Gamma$ for some word $u\in\wt{A}^*$. Now, by the inclusion of the two languages, we have that $uu^{-1}\in L(\Lambda, y_0)$, hence there is a walk $p$ in $\Lambda$ starting at $y_0$ labeled by $uu^{-1}$. By the determinism of $\Lambda$ we concluide that $p=y_0\mapright{u}z\mapright{u^{-1}}y_0$ is actually a circuit. Define the map $\psi:\Gamma\to\Lambda$ by putting $\psi(v)=z$. It is a well define map since if $x_0\mapright{w}v$ is another walk connecting $x_0$ with $v$, then $uw^{-1}\in L(\Gamma, x_0)\subseteq L(\Lambda, y_0)$ which in turns implies that $y_0\mapright{u}z'\mapright{w^{-1}}y_0$ is a circuit in $\Lambda$. Hence, by the determinism of $\Lambda$, we conclude that $z'=z$. By the way $\psi$ is defined, it immediately follows that $\psi$ is a morphism of inverse graphs with $\psi(x_0)=y_0$. Conversely, any morphism $\psi:\Gamma\to\Lambda$ with $\psi(x_0)=y_0$ implies that any word $w$ labeling a circuit $p$ centered at $x_0$, labels also the circuit $\psi(p)$ centered at $y_0$. 
\\
The equality $L(\Gamma,x_0)= L (\Lambda, y_0)$ implies that there are two morphisms $\psi: \Gamma\to\Lambda$, $\varphi:\Lambda\to \Gamma$ with $\psi(x_0)=y_0$, $\varphi(y_0)=x_0$. Since the composition $\varphi\circ \psi: \Gamma\to \Gamma$ is a morphism with $\varphi\circ \psi(x_0)=x_0$ by Lemma~\ref{lem: uniqueness of automorphism} we conclude that $\varphi\circ \psi$ is the identity morphism on $\Gamma$. Similarly,  $ \psi\circ \varphi$ is the identity morphism on $\Lambda$.
\end{proof}
Note that by the previous proposition we have that, up to isomorphism, there is a unique rooted inverse graph accepting the language $L(\Gamma,x_0)$. 
The set of automorphisms of the inverse graph $\Gamma$ will be denoted by $\Aut(\Gamma)$. This is a group acting on $\Gamma$ in a canonical way. In case the quotient graph $\Aut(\Gamma)\setminus\Gamma$ is finite, we say that $\Gamma$ is quasi-transitive. Note that this last condition is equivalent to the finiteness of $V(\Gamma)\setminus \Aut(\Gamma)$.


\subsection{Coverings of involutive and inverse graphs}\label{sec: coverings of inverse graphs}
In this section, we recall some basic topological facts of involutive graphs, and we derive some basic results for inverse graphs that will be used later in the proof of our characterizations.
\\
 Let $\Gamma$ be an involutive graph, for any $v\in V(\Gamma)$, let $\Star(\Gamma,v)=\{e\in E(\Gamma): \iota(e)=v\}$ be the star set of $\Gamma$ at $v$. Any morphism $\varphi:\Gamma\to\Gamma'$ between two involutive graphs induces a map $\varphi:\Star(\Gamma,v)\to \Star(\Gamma',\varphi(v)) $ between star sets in the obvious ways. Following Stallings \cite{Stall} we say that a graph morphism $\varphi$ is a \emph{cover} (\emph{immersion}) if the restrictions $\varphi:\Star(\Gamma,v)\to \Star(\Gamma',\varphi(v)) $ are bijections (injections) for every $v\in V(\Gamma)$. In his paper \cite{Stall}, Stallings made use of immersions between finite graphs to study finitely generated subgroups of free groups. Subsequently, Margolis and Meakin \cite{MaMe} showed how the theory of inverse monoids may be used to classify immersions between connected graphs, see \cite{NoraMe} for further extensions to higher dimensional cell complexes. By Lemma~\ref{lem: there are reduced walks} every morphism between two (complete) inverse $\wt{A}$-digraphs is a (covering) immersion. 
One central feature of coverings is the possibility of lifting walks, this is contained in the following characterization. 
\begin{prop}\cite[Proposition 2.3]{Gr-Meak}\label{prop: lifting of coverings}
Let $\varphi:\Gamma'\to\Gamma$ be an immersion between involutive connected graphs. Then, $\varphi$ is a cover if and only if, for each vertex $v\in V(\Gamma)$ and vertex $v'\in \varphi^{-1}(v)$, every walk $p$ in $\Gamma$ with $\iota(p)=v$ lifts to a unique walk $p'$ starting at $v'$, i.e., $\varphi(p')=p$. 
\end{prop}
An immediate consequence of the previous proposition is that covers of inverse graphs are surjective. We now briefly summarize some basic facts linking group theory and covers of graphs. Following \cite{Stall} two walks $p, q$ of a connected involutive graph $\Gamma$ are said to be homotopic equivalent (written $p\sim q$) if we may pass from $p$ to $q$ by a finite sequence of insertions or deletions of walks of the form $ee^{-1}$ for some edges of $e\in E(\Gamma)$; clearly $\iota(p)=\iota(q)$, $\tau(p)=\tau(q)$. Let $[p]$ denote the equivalence class (homotopy class) of the walk $p$ in $\Gamma$. Fix a vertex $v\in V(\Gamma)$, the set $\pi_1(\Gamma,v)=\{[p]: \iota(p)=\tau(p)=v\}$ together with the operation defined by $[p][q]=[pq]$ is a group, called the \emph{fundamental group of $\Gamma$ based at $v$}. If $\Gamma$ is an inverse $\wt{A}$-digraph, then for any $u\in\wt{A}^*$, there is at most one walk starting at $v$ with $u$ as label. Thus, it is routine to check that the map $\omega: \pi_1(\Gamma, v)\to \mathbb{F}_A$ sending $[p]$ to the element $\oo{\lambda(p)}\in \mathbb{F}_A$ is a well-defined map which is a monomorphism. With a slight abuse of notation we will denote by $\pi_1(\Gamma, v)$ the isomorphic image $\omega\left(\pi_1(\Gamma, v)\right)$. In this way $\pi_1(\Gamma, v)$ may be identified with the set $\oo{L(\Gamma, x_0)}$ of reduced words labeling a circuit at $x_0$. It is easily seen that the map $\theta: L(\Gamma, x_0)\to \pi_1(\Gamma,x_0)$ defined by $\theta(u)=\oo{u}$ is a (surjective) homomorphism of monoids. It is a well-known fact that the fundamental group of every connected involutive graph is free. 
By Proposition 2.3 of \cite[Chapter 23]{Handbook} and Proposition~\ref{prop: inclusion implies homomorphism}, the Stalling automaton of $\pi_1(\Gamma, v)$ is isomorphic to the \emph{core} of $(\Gamma, v)$, i.e., the (induced) rooted inverse subgraph obtained from $(\Gamma, v)$ by taking all the vertices belonging to a reduced circuit based at $v$. This operation is equivalent to removing all the ``hanging trees'' not containing $v$ from $(\Gamma, v)$. 
\\
The notion of determinizing $V$-quotient, for short $DV$-quotient, of an inverse $\wt{A}$-digraph $\Gamma=(V, E, \iota,\tau, \lambda)$ plays an important role in the proof of the characterization contained in Theorem~\ref{theo: main theorem}. This notion appears quite often in combinatorial group and inverse semigroup theory, for instance in the folding operation involved in the construction of a Stallings automaton, and the more general setting of Sch\"{u}tzenberger automata, see \cite{Steph}. Henceforth, we consider a very specific kind of $DV$-quotiens, the ones generated by a subset of vertices.
\begin{defn}[$DV$-quotients generated by a subset of vertices]
Let $W\subseteq V(\Gamma)$ be a subset of vertices and define the relations $\rho_W$ on $V(\Gamma)$ by $x\,\rho_W\, y$ if there are walks $w_1\mapright{u}x$, $w_2\mapright{u}y$, for some $u\in \wt{A}^*$, $w_1, w_2\in W$. This is clearly an equivalence relation on $V(\Gamma)$ that is compatible with $\Gamma$ in the following sense: if $x_1\mapright{v} x_2$, $y_1\mapright{v} y_2$, $v\in\wt{A}^*$, are walks in $\Gamma$ with $x_1\,\rho_W\,y_1$, then $x_2\,\rho_W\, y_2$. The $DV$-quotient $\Gamma/\rho_W$ is the $\wt{A}$-digraph with vertices the equivalence classes $V(\Gamma)/\rho_W$ and edges $[x]_{\rho_W}\mapright{a}[y]_{\rho_W}$ whenever $x_1\mapright{a}y_1$ is and edge in $\Gamma$ with $x_1\,\rho_W\, x$, $y_1\,\rho_W\,y$. We say that $W\subseteq V(\Gamma)$ has \emph{finite index} if the quotient graph $\Gamma/\rho_W$ is finite.
\end{defn}
The compatibility of $\rho_W$ ensures that the determinism of $\Gamma$ implies the determinism of $\Gamma/\rho_W$, hence 
$\Gamma/\rho_W$ is also an inverse graph. Let $\pi:\Gamma\to \Gamma/\rho_W$ be the morphism of inverse graphs defined by $\pi_v(x)=[x]_{\rho_W}$, $\pi_e(x\mapright{a}y)=[x]_{\rho_W}\mapright{a}[y]_{\rho_W}$. We have already remarked that the morphism $\pi:\Gamma\to\Gamma/\rho_{\rho_W}$ is an immersion. However, it is a cover if one restricts to subsets of vertices of an $\Aut(\Gamma)$-orbit.
\begin{lemma}\label{lem: it is a covering}
Let $H\le\Aut(\Gamma) $ be a subgroup of $\Aut(\Gamma)$. If $W\subseteq H\cdot x_0$, then, the morphism $\pi:\Gamma\to \Gamma/\rho_W$ is a cover of $ \Gamma/\rho_W$. Moreover, if $H\cdot x_0= W$, then for any $x\in V(\Gamma)$ we have $[x]_{\rho_W}=H\cdot x$.
\end{lemma}
\begin{proof}
First note that if $x \,\rho_W \,y$ then, there is an automorphism $\varphi\in H$ with $\varphi(x)=y$. Indeed, by definition there are paths $w_1\mapright{u}x$, $w_2\mapright{u}y$ with $w_1,w_2\in W$, and since $W\subseteq  H\cdot x_0$, there is an automorphism $\varphi\in H$ with $\varphi(w_1)=w_2$ which in turn implies $\varphi(x)=y$. In particular, we have shown that $[x]_{\rho_W}\subseteq H\cdot x$. Thus, the fact that two $\rho_W$-related elements are in the same $H$-orbit, implies that $|\Star(x,\Gamma)|=|\Star([x]_{\rho_W},\Gamma/\rho_W)|$,  from which we conclude that $\pi: \Star(x,\Gamma)\to \Star([x]_{\rho_W},\Gamma/\rho_W)$ is a bijection. 
\\
For the second statement, we have to show the other inclusion $H\cdot x\subseteq [x]_{\rho_W}$. Let $y\in H\cdot x$ then there is an automorphism $\varphi\in H$ with $y=\varphi(x)$ and so if $x_0\mapright{v} x$ is a walk connecting $x_0$ to $x$, this walk is sent into the walk $\varphi(x_0)\mapright{v}\varphi(x)=y$, with $\varphi(x_0)\in W$, i.e., $x\,\rho_W\, y$. 
\end{proof}
From this last Lemma~\ref{lem: it is a covering} and Proposition~\ref{prop: lifting of coverings} we immediately derive the following fact.  
\begin{lemma}\label{lem: lifting}
Let $W=H\cdot x_0$, for some $x_0\in V(\Gamma)$ and $H\le \Aut(\Gamma)$. Then, any circuit $[w]_{\rho_W}\mapright{v} [w]_{\rho_W}$ in $\Gamma/\rho_W$ with $w\in W$, lifts uniquely to a walk $w\mapright{v} w_1$ in $\Gamma$ for some $w_1\in W$.
\end{lemma}
We have the following lemma.
\begin{lemma}\label{lem: finite index}
Let $W=H\cdot x_0$, for some $x_0\in V(\Gamma)$ and $H\le \Aut(\Gamma)$. Then, $H\le \Aut(\Gamma)$ is a finite index subgroups if and only if $\Gamma/\rho_W$ is finite. 
\end{lemma}
\begin{proof}
By Lemma~\ref{lem: uniqueness of automorphism} we have $(H\varphi)\cdot x_0\cap (H\psi)\cdot x_0\neq\emptyset$ if and only if $\varphi\psi^{-1}\in H$. Hence, by Lemma~\ref{lem: it is a covering} we have $[\varphi(x_0)]_{\rho_W}= [\psi(x_0)]_{\rho_W}$ if and only if $ H\varphi=H\psi$. Thus, different right cosets give rise to different $\rho_W$-classes, and these are finitely many if and only if there are finitely many right cosets. 
\end{proof}
\begin{prop}\label{prop: language inverse homomorphism}
Let $W=H\cdot x_0$, for some $x_0\in V(\Gamma)$ and $H\le \Aut(\Gamma)$. Then, there is a surjective homomorphism of monoids:
\[
\eta:L( \Gamma/\rho_W,[x_0]_{\rho_W})\to H
\]
such that $L(\Gamma, x_0)=\eta^{-1}(\mathds{1}_H)$, where $\mathds{1}_H$ is the identity of $H$. Furthermore, there is a surjective homomorphism 
\[
\oo{\eta}: \pi_1(\Gamma/\rho_W, [x_0]_{\rho_W})\to H
\]
with $\oo{\eta}^{-1}(\mathds{1}_H)=\oo{L(\Gamma,x_0)}$.
\end{prop}
\begin{proof}
By Lemma~\ref{lem: lifting} every circuit $[x_0]_{\rho_W}\mapright{u}[x_0]_{\rho_W}$ in $\Gamma/\rho_W$ may be uniquely lifted to a walk $x_0\mapright{u}w$ for some $w\in W$. Since $W=H\cdot x_0$, there is a (unique by Lemma~\ref{lem: uniqueness of automorphism}) automorphism $\varphi_u\in H$ sending $x_0$ to $w$. Let us define the morphism $\eta$ by putting $\eta(u)=\varphi_u$. By the uniqueness of the automorphism $\varphi_u$, $\eta$ is a well-defined function. Let us show that it is also a homomorphism. Take two circuits 
\[
[x_0]_{\rho_W}\mapright{u}[x_0]_{\rho_W}, \quad [x_0]_{\rho_W}\mapright{v}[x_0]_{\rho_W} \mbox{ in }\Gamma/\rho_W
\]
and lift them to the walks $x_0\mapright{u}w_1$, $x_0\mapright{v}w_2$ of $\Gamma$ for some $w_1, w_2\in W$ and let $\varphi_u, \varphi_v\in H$ be two automorphisms with $\varphi_u(x_0)=w_1$, $\varphi_v(x_0)=w_2$. Similarly, the composition of the two circuits $[x_0]_{\rho_W}\mapright{uv}[x_0]_{\rho_W}$ lifts uniquely to the walk  $x_0\mapright{uv}w_3$ for some $w_3\in W$ and let $\varphi_{uv}\in H$ be the corresponding automorphism with $\varphi_{uv}(x_0)=w_3$. Now, by applying the automorphism $\varphi_u$ to the walk $x_0\mapright{v} \varphi_v(x_0)$ we get the walk $\varphi_u(x_0)\mapright{v}\varphi_u(\varphi_v(x_0))$. Hence, since $w_1=\varphi_u(x_0)$ in $\Gamma$ there is the walk $x_0\mapright{uv}\varphi_u(\varphi_v(x_0))$. Therefore, by the determinism of $\Gamma$ we deduce 
\[
\varphi_u(\varphi_v(x_0))=w_3=\varphi_{uv}(x_0)
\]
and so by Lemma~\ref{lem: uniqueness of automorphism} $\varphi_u\circ \varphi_v=\varphi_{uv}$. From this last equation we conclude that $\eta$ is a homomorphism $\eta(uv)=\varphi_{uv}=\varphi_u\circ \varphi_v=\eta(u)\circ \eta(v)$. Let us show that $\eta$ is surjective. For any $\varphi\in H$ take any walk $x_0\mapright{u}\varphi(x_0)$, this walk maps onto a walk $[x_0]_{\rho_W}\mapright{u}[\varphi(x_0)]_{\rho_W}$ with $\eta(u)=\varphi$. The set $\eta^{-1}(\mathds{1}_H)$ is formed by the circuits $[x_0]_{\rho_W}\mapright{u}[x_0]_{\rho_W}$  whose lifted walks $x_0\mapright{u}w$ have the property that $w=x_0$, i.e., $u\in L(\Gamma, x_0)$. For the last statement of the proposition put $\oo{\eta}=\eta(\theta^{-1})$ where the map $\theta: L( \Gamma/\rho_W,[x_0]_{\rho_W})\to \pi_1(\Gamma/\rho_W, [x_0]_{\rho_W})$ is the surjective homomorphism defined by $\theta(u)=\oo{u}$. This is a well-defined homomorphism: take any two words $u_1,u_2\in\theta^{-1}(g)$ for some $g\in  \pi_1(\Gamma/\rho_W, [x_0]_{\rho_W})$. By definition of $\eta$ there are two automorphisms $\varphi_1, \varphi_2\in H$ with $\eta(u_1)=\varphi_1$, $\eta(u_2)=\varphi_2$ and two walks $x_0\mapright{u_1}\varphi_1(x_0)$, $x_0\mapright{u_2}\varphi_2(x_0)$. Now, since $\Gamma$ is inverse by Lemma~\ref{lem: there are reduced walks} we have that there are also the walks $x_0\mapright{\oo{u_1}}\varphi_1(x_0)$, $x_0\mapright{\oo{u_2}}\varphi_2(x_0)$. Hence, since $\oo{u_1}=\oo{u_2}=g$, by Lemma~\ref{lem: uniqueness of automorphism} we conclude that $\varphi_1=\varphi_2$, i.e., $\eta(u_1)=\eta(u_2)$. 
\end{proof}


\section{Context-free grammars, pushdown automata and configurations graphs}

There are several ways to define a context-free language, either via the notion of grammar, or using the notion of pushdown automaton. For the sake of completeness we briefly recall the main notions that are used in the paper, and we refer the reader to \cite{hop} for further details. A \emph{context-free grammar} is a tuple $\mathcal{G}=(V,A,P,S)$ where $V$ is a finite set of non-terminal symbols (or variables), $A$ is the finite alphabet of terminal symbols that is disjoint from $V$, $S\in V$ is the starting symbol, and  $P\subseteq V\times (V\cup A)^*$ is a finite set called productions. Usually production are denoted by $X\Rightarrow_{\mathcal{G}} \alpha$, $\alpha\in (V\cup A)^*, X\in V$.  We extend $\Rightarrow_{\mathcal{G}} $ by putting $\alpha X \gamma\Rightarrow_{\mathcal{G}} \alpha\beta\gamma$ whenever $X\Rightarrow_{\mathcal{G}} \beta\in P$, and we denote by $\Rightarrow_{\mathcal{G}} ^*$ the reflexive and transitive closure of $\Rightarrow_{\mathcal{G}} $. For a variable $X\in V$ the language generated by $X$ is the set
\[
L(X)=\{u\in A^*: X\Rightarrow_{\mathcal{G}}^* u\}
\]
the language generated by the grammar $\mathcal{G}$ is $L(S)$, and in this case we say that it is a context-free language. It is a well-known fact that every context-free grammar $\mathcal{G}$ can be replaced by one in Chomsky normal form:
\begin{prop}[Theorems 4.5 and 4.3 \cite{hop}]
Any context-free language $L$ with $\varepsilon\notin L$ is generated by a grammar $\mathcal{C}$ in which all productions are of the form $X\Rightarrow_{\mathcal{C}} YZ$, with $Y,Z$ variables different from the starting symbol, and $X\Rightarrow_{\mathcal{C}} a$, with $a\in A$ terminal symbol. In case, $\varepsilon\in L$ there is just one production involving $\varepsilon$ which is $S\Rightarrow_{\mathcal{C}}\varepsilon$.
\end{prop}
We call a rule of the form $X\Rightarrow_{\mathcal{C}} YZ$ \emph{non-terminal}, while one of the form $X\Rightarrow_{\mathcal{C}} a$, \emph{terminal}. Henceforth, we assume that all the variable of $\mathcal{G}$ are useful, in the sense that any $X\in V$ appears in some derivation $S\Rightarrow_{\mathcal{G}}^* w $ for some $w\in L(\mathcal{G})$.
\begin{rem}\label{rem: last rules with terminal}
In the production $S\Rightarrow_{\mathcal{C}} w$ of a word $w$ it is not difficult to see that we may obtain the same word $w$ by first applying all the non-terminal rules followed by the terminal ones.
\end{rem}
An alternative way to describe such languages is via the notion of pushdown automaton (PDA). A PDA is a tuple $M=(Q,\Sigma,\Theta,\delta,q_0,\perp,F)$, where $Q$ is the set of states, $\Sigma$ is the alphabet of the input tape, $\Theta$ is the alphabet of the stack, $q_0$ is the initial state, $\perp\in\Theta$ is the initial symbol of the stack, and $F\subseteq Q$ is the set of final states. The transition function $\delta$ is a map from $ Q\times (\Sigma\cup \{1\})\times \Theta$  to subsets of $Q\times \Theta^*$. The interpretation of $\delta(q,a,x)=\{(p_1,\gamma_1),\ldots, (p_n, \gamma_n)\}$ is that when $M$ is at the state $q$ and is reading $a$ in the input and $x$ is the last symbol on the stack, it can nondeterministically choose to enter in state $p_i$ and replace the letter $x$ with the word $\gamma_i\in\Theta^*$ for some $i\in [1,n]$; in case $\gamma_i=1$, we are removing the symbol $x$ from the top of the stack (pop operation), otherwise we are pushing some word. Throughout the paper we will describe $\delta$ by using the notation $(q,a,x) \sststile{M}{}(p_i,\gamma_i)$ to denote that $(p_i,\gamma_i)\in \delta(q,a,x)$. Transitions of the form $(q,1,x) \sststile{M}{}(p_i,\gamma_i)$ are called $1$-moves. The pushdown automaton $M$ is called \emph{deterministic} (DPDA) when for each $(q,a,x)\in  Q\times \Sigma\times \Theta$ there is at most one $(p,\gamma)\in Q\times \Theta^*$ such that $(q,a,x) \sststile{M}{}(p,\gamma)$ is a transition, and if there is a $1$-move $(q,1,x) \sststile{M}{}(p,\gamma)$ then $\delta(q,a,x)$ is empty for all $a\in \Sigma$. 
\\
The dynamics of $M$ is described by the \emph{instantaneous description} that defines how the configuration of the machine evolves at each step of the computation. We write $(q,\xi x) \sdtstile{M}{a}(p,\xi \gamma )$ (sometimes $(q,\xi x) \sdtstile{\delta}{a}(p,\xi \gamma )$ when we want to emphasize which transition function we are considering) if we have $(q,a,x)\sststile{M}{}(p,\gamma)$. If there is a sequence 
\[
(q_1,\xi_1) \sdtstile{M}{a_1}(q_2,\xi_2 ) \sdtstile{M}{a_2}\ldots  \sdtstile{M}{a_n}(q_{n+1},\xi_{n+1} )
\]
for some $w=a_1\ldots a_n\in \Sigma^*$, we write $(q_1,\xi_1)\sdtstile{M}{w} (q_{n+1},\xi_{n+1} )$. 
The \emph{language accepted by} $M$ is the set $L(M)=\{w\in \Sigma^*: (q_0,\perp)\sdtstile{M}{w} (p,\xi)\mbox{ with }p\in F \}$. One may define the language accepted by $M$ by requiring that the stack contains only the initial symbol $\perp$, so, in this case, $L_e(M)=\{w\in \Sigma^*: (q_0,\perp)\sdtstile{M}{w} (p, \perp) \}$. It is not difficult to see, by a standard argument similar to \cite[Theorem~5.1 and Theorem~5.2]{hop}, that these two models of computation are equivalent in the sense that these two model of acceptance define the same class of languages. Throughout the paper we will use the first model, although the condition of the stack containing the initial symbol will be used in some later proofs. 
\\
Following \cite{muahu2} we may associate to $M$ a $\Sigma\cup\{1\}$-directed graph $\Gamma(M)$, called the \emph{configurations graph}, that is obtained as follows: the set of vertices $V(\Gamma(M))$ is the subset of $Q\times \Theta^*$ containing $(q_0,\perp)$ and all the pairs $(p,\xi)$ for which there is a computation $(q_0,\perp)\sdtstile{M}{w} (p,\xi )$ for some $w\in \Sigma^*$, and we draw an oriented edge $(q_1,\xi_1)\mapright{a}(q_2,\xi_2)$, $a\in \Sigma\cup\{1\}$ between two vertices of $V(\Gamma(M))$ whenever we have the computation $(q_1,\xi_1) \sdtstile{M}{a}(q_2,\xi_2)$ for some $a\in \Sigma\cup\{1\}$; note that we may have in general edges labeled with the empty word $1$. This graph may be used as an acceptor for the language accepted by $M$ since $L(M)$ is formed by all the words that label a walk connecting the initial configuration $(q_0, \perp)$ to a configuration $(p,\xi)$ which is in a final state $p\in F$. 

\subsection{Inverse configurations graphs, inverse pushdown automata, and inverse-context-free languages}
In what follows, we will concentrate our attention on deterministic pushdown automata that are real-time, so without $1$-moves. Pushdown automata without $1$-moves have configurations graphs without edges labeled with the empty word $1$, so it is the natural framework when considering configurations graphs that are inverse $\Sigma$-directed graphs. It is clear that to have $\Gamma(M)$ that is a deterministic $\Sigma$-directed graph, we need to assume $M$ deterministic. The natural question is when $\Gamma(M)$ is also co-deterministic, since in this case we may transform $\Gamma(M)$ into an inverse graph $\Gamma(M)^-$ by simply adding for every edge $(q,\xi x)\mapright{a}(p,\xi\gamma)$ the reverse one $(p,\xi\gamma) \mapright{a^{-1}}(q,\xi x)$. In particular, it is of interest to see when $\Gamma(M)^-$ is the configurations graph of some  deterministic PDA. We shall see that this happens if $M$ has a reverse transition function, so $M$ is a reversible pushdown automaton, see \cite{kutrib}. Reversibility gives the possibility of stepping the computation back and forth, and it is a central notion in the theory of computation and quantum computing, for the specific case of reversible PDA we refer the reader to \cite{kutrib}. A reversible pushdown automaton $M=(Q,\Sigma, \Theta, \delta, q_0,\perp, F)$ is characterized by having a reverse transition function $\delta_R: Q\times \Sigma \times \Theta\to Q\times \Theta^*$ such that is able to locally reverse the computation, i.e., 
\[(q_1,\xi_1)\sdtstile{\delta}{a}(q_2,\xi_2)\text{ if and only if }(q_2,\xi_2)\sdtstile{\delta_R}{a}(q_1,\xi_1).\] 
We remark that the initial assumption that $M$ is real-time is not strict from the language point of view since by \cite[Theorem 5]{kutrib} for every reversible pushdown automaton $N$ there is a real-time reversible PDA $M$ such that $L(M)=L(N)$. Reversible pushdown automata have a peculiar dynamics since by \cite[Fact 3]{kutrib} no transition changes the length of the stack by more than one. Hence, each transition in $\delta$ has one of the following form
\[
(q,a,x) \sststile{M}{}(p,xy),\,\text{or } (q,a,x) \sststile{M}{}(p,x),\,\text{ or }(q,a,x) \sststile{M}{}(p,1)
\]
for some $x, y\in\Theta$, corresponding to the following reverse transitions
\[
(p,a,y) \sststile{\delta_R}{}(q,1),\,\text{or } (p,a,x) \sststile{\delta_R}{}(q,x),\,\text{or } (p,a,z) \sststile{\delta_R}{}(q,zx)
\]
for some $x, y, z\in\Theta$. 
\\
To deal with inverse graphs, we need to introduce a new class of PDAs more suitable to consider languages on an involutive alphabet $\wt{\Sigma}$. 
\begin{defn}[Inverse PDA]
A real-time DPDA $M=(Q,\wt{\Sigma}, \Theta, \delta, q_0,\perp, F)$ is called \emph{inverse} if  $\delta$ may be partitioned into two functions: $\delta^+$ which is the restriction of $\delta$ to $Q\times \Sigma\times \Theta$ and $\delta^-$ which is the restriction of $\delta$ to $Q\times \Sigma^{-1}\times \Theta$ such that $\delta^+$ is the reverse transition of $\delta^-$ and vice versa in the following sense:
\[
(q_1,\xi_1)\sdtstile{\delta^+}{a}(q_2,\xi_2)\text{ if and only if }(q_2,\xi_2)\sdtstile{\delta^-}{a^{-1}}(q_1,\xi_1).
\]
A language $L\subseteq \wt{\Sigma}^*$ is called inverse-context-free if $L=L(M)$ for some inverse PDA $M$ on $\wt{\Sigma}$.
\end{defn}
It is clear from the definition and the fact that $\delta^+, \delta^-$ are functions, that $\Gamma(M)$ is inverse if and only if $M$ is an inverse PDA. An important property of inverse PDAs that will be useful later, is the fact that in the configurations graph $\Gamma(M)$ there are no configurations with the stack that is empty. Indeed, any computation $(p,x)\sdtstile{M}{a}(q, 1)$ should have the reverse computation $(q, 1)\sdtstile{M}{a^{-1}}(p,x)$ which is clearly not allowed in the definition. For an invertible PDA $M$ let us set $M^{+}=(Q,\Sigma, \Theta, \delta^+, q_0, \perp, F)$.
There is a strict relationship between reversible and invertible PDAs. Indeed, any reversibile real-time PDA $N=(Q,\Sigma, \Theta, \delta, q_0,\perp, F)$ with reverse transition function $\delta_R$ may be ``extended'' to an invertible PDA $\wt{N}=(Q,\wt{\Sigma}, \Theta, \wt{\delta}, q_0,\perp, F)$ on the involutive alphabet $\wt{\Sigma}$, by extending the map $\delta$ to the map $\wt{\delta}$ defined by $\wt{\delta}(q,a^{-1}, x)=(p,\gamma)$ whenever $\delta_R(q,a,x)=(p,\gamma)$. It immediately follows from the definitions that $\wt{N}$ is inverse, $\wt{N}^+=N$, and $\Gamma(N)^-=\Gamma(\wt{N})$. By this last construction we deduce that a PDA $M$ is inverse if and only if $M^+$ is reversible. Moreover, it is evident from the language theoretic point of view, that inverse PDAs and reversible PDAs are essentially the same class of PDAs, since for any language $L$ accepted by a real-time reversible PDA $N$, the associated inverse PDA $\wt{N}$ has the property $L=L(\wt{N})\cap \Sigma^*$. 

\section{Context-free inverse graphs}\label{sec: context-free graphs}
In their paper \cite{muahu}, Muller and Schupp called a finitely generated group $G$ presented by $\la A\mid \mathcal{R}\ra$ context-free if the word problem $WP(G;A)$ is a context-free language. With our notation, this language coincides with the set $L(\Cay(G;A),\mathds{1}_G)$ of words labeling a circuit starting and ending at the group identity, so the following definition is a natural generalization for a generic inverse graph.
\begin{defn}
A rooted inverse graph $(\Gamma, x_0)$ is called context-free if $L(\Gamma, x_0)$ is a context-free language.
\end{defn}
The aim of this section is to show that for an inverse graph $\Gamma$ being context-free is equivalent to be context-free in the sense of Muller-Schupp that is also equivalent to be the configurations graph of an inverse PDA.
To prove this characterization we first need to show that the rooted graph $(\Gamma, x_0)$ is a context-free graph in the sense of Muller and Schupp, a fact that does not immediately follow from the context-freeness of the language $L(\Gamma, x_0)$. This fact has been already proved in \cite[Theorem 4.6]{Ce-Wo} in the complete case (fully deterministic in their notation), we generalize this fact for any context-free inverse graph. First we need to make some geometric considerations that arise from the fact that the language accepted by such a graph is generated by a context-free grammar $\mathcal{G}$. From now on we assume that the grammar $\mathcal{G}=(V,A,P,S)$ generating $L(\Gamma,x_0)$ is in Chomsky normal form. To visualize the productions applied to the language $L(\Gamma, x_0)$, it is useful to introduce the notion of  $\mathcal{G}$-triangulation and $\mathcal{G}$-decoration of the inverse graph $\Gamma$. This is another more convenient way to adapt the notion of $K$-triangulation, introduced in \cite{muahu, muahu2} for Cayley graphs, to the general case of an inverse graph. Similar notions have been also considered in several other papers, see for instance \cite{Ce-Wo, Pele, Wo89}.
\begin{figure}[htp]
\includegraphics[scale=1.5]{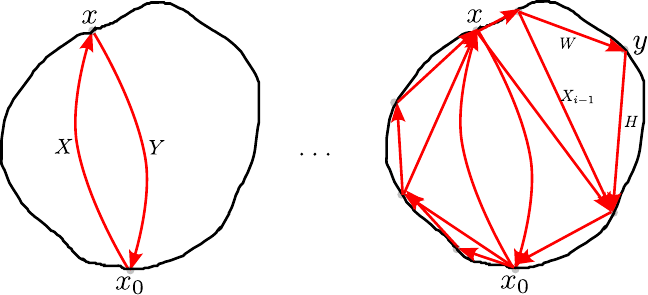}
\caption{The first step and the generic $i$-th step of the triangulation of a circuit $p=x_0\mapright{w} x_0$.} \label{fig: triang}
\end{figure}
\begin{defn}[$\mathcal{G}$-triangulation and $\mathcal{G}$-decoration]\label{defn: triangulation}
Let $w\in L(\Gamma, x_0)$, and let $p=x_0\mapright{w} x_0$ be the corresponding circuit in $\Gamma$. By Remark~\ref{rem: last rules with terminal}, any production $S\Rightarrow^*_{\mathcal{C}}w$ can be obtained by first applying all the non-terminal rules followed by the terminal ones. A $\mathcal{G}$-triangulation of $p=x_0\mapright{w} x_0$ associated to the production $S\Rightarrow^*_{\mathcal{C}}w$, is obtained by adding to $p$ new edges with labels in the variable set $V$, let us call them $V$-edges. Suppose the first production of $S\Rightarrow^*_{\mathcal{C}}w$ is $S\Rightarrow_{\mathcal{C}} XY$, where $X\Rightarrow^*_{\mathcal{C}} w_1$, $Y\Rightarrow^*_{\mathcal{C}} w_2$ with $w=w_1w_2$. In the first step we add to $p$ the two new edges $e_1, g_1$ with $\iota(e_1)=\tau(g_1)=x_0$, $\tau(e_1)=\iota(g_1)=x$ where $x\in p$ is the vertex for which $p$ factorizes as $p=x_0\mapright{w_1} x\mapright{w_2} x_0$. We also label $e_1, g_1$ by $\lambda(e_1)=X$, $\lambda(g_1)=Y$ and we say that the edges $e_1, g_1$ will give rise to the walks $x_0\mapright{w_1}x$, $x\mapright{w_2}x_0$, respectively, see Fig. \ref{fig: triang}. Now, iteratively, at step $i$ of the production 
\[
S\Rightarrow^{i-1}_{\mathcal{C}} \alpha_{i-1}X_{i-1}\beta_{i-1}\Rightarrow_{\mathcal{C}}\alpha_{i-1}WH\beta_{i-1}\Rightarrow_{\mathcal{C}}^*w
\]
we have applied rule $X_{i-1}\Rightarrow_{\mathcal{C}} WH$, then in $p$ we have already an edge $h_{i-1}$ with $\lambda(h_{i-1})=X_{i-1}$ which will give rise to the subwalk $q=\iota(h_{i-1})\mapright{u} \tau(h_{i-1})$ of $p$ and rule $X_{i-1}\Rightarrow_{\mathcal{C}} WH$ corresponds to the factorization $\iota(h_{i-1})\mapright{u_1} y\mapright{u_2} \tau(h_{i-1})$ of $q$ given by the productions $W\Rightarrow_{\mathcal{C}}^* u_1$, $H\Rightarrow_{\mathcal{C}}^* u_2$, $u=u_1 u_2$. Then, we create two new edges $e_i, g_i$ with $\iota(e_i)=\iota(h_{i-1}), \tau(g_i)=\tau(h_{i-1})$, $\tau(e_i)=\iota(g_i)=y$, $\lambda(e_i)=W$, $\lambda(g_i)=H$. We say that $e_i,g_i$ are generated from $h_{i-1}$. The process of triangulation stops when we apply the first terminal rule. 
The $\mathcal{G}$-decoration of $\Gamma$ is the $(\wt{A}\cup V)$-directed graph $\Gamma_{\mathcal{G}}$ obtained from $\Gamma$ by adding all the new $V$-edges involved in all the possible $\mathcal{G}$-triangulations of all the circuits starting at $x_0$.
\end{defn}
Throughout the paper the following constant
\[
K=\max\left\{\min\{|u|:u\in L(X), X\in V\}\right \}
\]
will play a fundamental role. The following lemma can be found in \cite{muahu} in the case of Cayley graphs of finitely generated groups, but also in \cite[Lemma 4.5]{Ce-Wo} with a similar spirit. 
\begin{lemma}\label{lem: existence of derived walks}
Let $e$ be a $V$-edge of $\Gamma_{\mathcal{G}}$ with $\lambda(e)=X\in V$. Then, for any $z\in L(X)$, there is a walk $\iota(e)\mapright{z}\tau(e)$ in $\Gamma$. In particular, the distance in $\Gamma$ between the initial and terminal vertex of $e$ satisfies the bound $d(\iota(e), \tau(e))\le K$.
\end{lemma}
\begin{proof}
By definition of $\mathcal{G}$-decoration there are derivations $S\Rightarrow_{\mathcal{C}}^* W_1XW_2$, $W_1, W_2\in V^*$, $W_1\Rightarrow_{\mathcal{C}}^* w_1, W_2\Rightarrow_{\mathcal{C}}^* w_2, X\Rightarrow_{\mathcal{C}}^* u$, $w_1,w_2,u\in\wt{A}^*$, corresponding to a circuit $x_0\vlongmapright{w_1uw_2} x_0$ in $\Gamma$ where $x_0\mapright{w_1}\iota(e)$, $\tau(e)\mapright{w_2}x_0$ are two subwalks. Since $z\in L(X)$ there is a derivation $X\Rightarrow_{\mathcal{C}}^*  z$, and so also a derivation $S\Rightarrow_{\mathcal{C}}^* w_1zw_2$. Thus, since $\Gamma$ is deterministic and $w_1zw_2\in L(\Gamma,x_0)$ we deduce that $\iota(e)\mapright{z}\tau(e)$ is also a walk in $\Gamma$.
\end{proof}
From this lemma we deduce that the decorated graph $\Gamma_{\mathcal{G}}$ is deterministic: if $e,g$ are two new edges with $\iota(e)=\iota(g)$ and $\lambda(e)=\lambda(g)=X$, then for any $z\in L(X)$ there are walks $\iota(e)\mapright{z}\tau(e)$, $\iota(g)\mapright{z}\tau(g)$ in $\Gamma$, hence the determinism and condition $\iota(e)=\iota(g)$ yield $\tau(e)=\tau(g)$.
Henceforth, $d(x,y)$ denotes the distance in the inverse graph $\Gamma$, and put $\|y\|_{x_0}=d(x_0,v)$. In \cite{muahu2} the Muller and Schupp have introduced the central notion of end-cone together with the idea of end-isomorphism of such subgraphs. We briefly recall them. For a vertex $v$ with $\|v\|_{x_0}=n$, the end-cone at $v$ is the (inverse) subgraph $\Gamma(v,x_0)$ that is the connected component of $\Gamma\setminus D_{n-1}(x_0)$ containing $v$, and we denote by $\Delta(v,x_0)=\{y\in \Gamma(v,x_0): \|y\|_{x_0}=n\}$ the set of frontier points of $\Gamma(v,x_0)$. An \emph{end-isomorphism} $\psi:\Gamma(v,x_0)\to \Gamma(w,x_0)$ between the end-cones $\Gamma(v,x_0), \Gamma(w,x_0)$ is an isomorphism of inverse graphs with the property of preserving the frontier points: $\psi(\Delta(v,x_0))=\Delta(w,x_0)$. The rooted graph $(\Gamma, x_0)$ is called \emph{a context-free graph} if up to end-isomorphism there are finitely many end-cones. In \cite{muahu2} it is shown that the graph $(\Gamma(M), x_0)$ of the instantaneous descriptions of a pushdown automaton is a context-free graph, and conversely given a context free graph $(\Gamma,x_0)$  there is a pushdown automaton $M$ such that $(\Gamma(M), x_0)$ is isomorphic to $(\Gamma, x_0)$. 
\\
Now we show that an inverse graph $\Gamma$ whose accepted language $L(\Gamma, x_0)$ is context-free, is also a context-free graph in the sense of Muller and Schupp. The strategy of the proof is to show that up to end-isomorphism there are finitely many end-cones $\Gamma(v,x_0)$ by showing that such a graph is completely determined by the finite decorated subgraph
\[
\Delta^C(v,x_0)=\la x\in V(\Gamma_{\mathcal{G}}):\exists y\in \Delta(v,x_0)\mbox{ with }d(x,y)\le C\ra_{\Gamma_{\mathcal{G}}}
\]
of the decorated graph $\Gamma_{\mathcal{G}}$, where $C\ge 1$ is a constant independent from $v$. To this end, we need to introduce some notations and concepts that are useful in the proof of such fact. We say that a walk $p=x\mapright{u}y$ contained in $\Gamma(v,x_0)$ is \emph{$C$-external}, or simply external when $C$ is clear from the context, if the only vertices belonging to $\Delta^C(v,x_0)$ are the endpoints $x,y$. For a walk $p=v_1\mapright{w}v_2\subseteq \Gamma(v,x_0)$ with $v_1,v_2\in\Delta^C(v,x_0)$, the set of \emph{external walks} $\mathcal{E}(p, C)$ is defined as the set of all the subwalks $\pi_1, \ldots, \pi_k$ of $p$ that are external, see Fig.~\ref{fig: external paths} for a graphical representation. We have the following lemma.
\begin{lemma}\label{lem: existence of v-edge}
Let $\ell=x_0\mapright{v_1}x\mapright{u}y\mapright{v_2}x_0$ be a circuit in $\Gamma$ such that 
$g_1=x_0\mapright{v_1}x$, $g_2=y\mapright{v_2}x_0$ are two geodesics connecting $x_0$ with two vertices $x,y\in \Delta(v,x_0)$ with $\|v\|_{x_0}=n>K$, and $p=x\mapright{u}y\subseteq \Gamma(v,x_0)$. Then, there is a $V$-edge $\hat{e}$ giving rise to the subwalk $\ell(\iota(\hat{e}), \tau(\hat{e}))$ containing $p$, such that the length of the subwalks $g_1(\iota(\hat{e}), x)$, $g_2(y,\tau(\hat{e}))$ is at most $K$. 
\end{lemma}
\begin{proof}
Consider a $\mathcal{G}$-triangulation of the circuit $\ell$ and let $\mathcal{T}$ be the set of $V$-edges occurring in such triangulation. Let us assume that the two geodesics may be decomposed in the following way
\begin{eqnarray*}
x_0\mapright{v_1}x=x_0\mapright{a_1}x_1\mapright{a_2}x_2\ldots \mapright{a_n}x_n\\
x_0\mapright{v_2^{-1}}y=y_0\mapright{b_1}y_1\mapright{b_2}y_2\ldots \mapright{b_n}y_n
\end{eqnarray*}
with $x_n=x, y_n=y$. In the first step of the $\mathcal{G}$-triangulation we have two $V$-edges $e_1, g_1$ giving rise to the subwalks $\ell(x_0, z), \ell(z,x_0)$, respectively, where $z=\tau(e_1)=\iota(g_1)$. Since any vertex $s$ of $p$ has the property $\|s\|_{x_0}\ge n$, we have that $z\notin p$, for if we would have $\|z\|_{x_0}\ge n$, hence by Lemma~\ref{lem: existence of derived walks} we would get $n\le \|z\|_{x_0}\le K$, a contradiction. One among $e_1, g_1$ gives rise to a subwalk of $\ell$ containing $p$. This implies that the set $\mathcal{T}'$ of $V$-edges $e$ of the triangulation $\mathcal{T}$ with $\iota(e)\in \{x_0,\ldots, x_n\}$, $\tau(e)\in \{y_0,\ldots, y_n\}$ and giving rise to a subwalk containing $p$, is non-empty. Choose in $\mathcal{T}'$ a $V$-edge $\hat{e}$ with the following maximality condition: if $x_i=\iota(\hat{e})$ and $y_j=\tau(\hat{e})$, then $i+j$ is maximal. Now, we claim that the $V$-edge $\hat{e}$ gives rise to two new edges $e', g'$ in the triangulation $\mathcal{T}$ with the property $\tau(e')=\iota(g')\in p$. Indeed, if for instance we would have $y_{j'}=\tau(e')\in \{y_0,\ldots, y_n\}$, then $j'<j$ would contradict the choice of $\hat{e}$.
Note that $\|\iota(\hat{e})\|_{x_0}\ge n-K$ holds, since otherwise $\|\iota(\hat{e})\|_{x_0}< n-K$ would imply $\|\tau(e')\|_{x_0}< \|\iota(\hat{e})\|_{x_0}+d(\iota(\hat{e}), \tau(e'))\le (n-K)+K=n$, by Lemma~\ref{lem: existence of derived walks} and the fact that $e'$ is a $V$-edge connecting $\iota(\hat{e})$ and $\tau(e')$, however, this contradicts condition $\tau(e')\in p$ (recall that all the vertices of $\Gamma(v,x_0)$ have the property of being at distance at least $n$ from $x_0$). Similarly, we have $\|\tau(\hat{e})\|_{x_0}\ge n-K$.
Since $g_1=x_0\mapright{v_1}x$, $g_2=y\mapright{v_1}x_0$ are two geodesics of length $n$ and $\|\iota(\hat{e})\|_{x_0},\|\tau(\hat{e})\|_{x_0}\ge n-K$ we conclude that the length of the subwalks $g_1(\iota(\hat{e}), x)$, $g_2(y,\tau(\hat{e}))$ is at most $K$. 
\end{proof}
From the previous lemma we immediately conclude that the diameter of $\Delta(v,x_0)$ is uniformly bounded from above, see for instance \cite[Lemma 17]{Pele}.
\begin{prop}\label{prop: frontier uniformly bounded}
For any $v\in V(\Gamma)$ the diameter $\delta_{\Gamma}(\Delta(v,x_0))\le 3K$.
\end{prop}
\begin{proof}
It follows from the previous Lemma~\ref{lem: existence of v-edge} and Lemma~\ref{lem: existence of derived walks} and the triangular inequality:
\[
d(x,y)\le d(x,\iota(\hat{e}))+d(\iota(\hat{e}), \tau(\hat{e}))+d(\tau(\hat{e}), y)\le 3K
\]
\end{proof}
We say that a morphism of $(\wt{A}\cup V)$-directed graphs $\psi:  \Delta^C(v,x_0)\to  \Delta^C(w,x_0)$ is \emph{norm-preserving} if for any $z_1, z_2\in  \Delta^C(v,x_0)$ with $\|z_1\|_{x_0}\le \|z_2\|_{x_0}$ implies $\|\psi(z_1)\|_{x_0}\le \|\psi(z_2)\|_{x_0}$. Note that we have that if $\|z_1\|_{x_0}= \|z_2\|_{x_0}$ then $\|\psi(z_1)\|_{x_0}= \|\psi(z_2)\|_{x_0}$. In case $\psi:  \Delta^C(v,x_0)\to  \Delta^C(w,x_0)$ is an isomorphism of $(\wt{A}\cup V)$-directed graphs we require that also $\psi^{-1}$ is norm-preserving. Therefore, in this case, since $\Delta^C(v,x_0)$ is finite we may conclude that for any $u\in\Delta(v,x_0)$, $\|\psi(u)\|_{x_0}=\|w\|_{x_0}$, i.e. $\psi(\Delta(v,x_0))=\Delta(w,x_0)$.
\begin{lemma}\label{lem: external walk}
Put $C\ge 2K$. Let $v,w$ be two vertices with $\|v\|_{x_0},\|w\|_{x_0}>K$ such that there is a norm-preserving isomorphism $\psi:  \Delta^C(v,x_0)\to  \Delta^C(w,x_0)$. Then, for any external walk $x'\mapright{u}x''$ in $\Gamma(v,x_0)$ there is an external walk $\psi(x')\mapright{u}\psi(x'')$ in $\Gamma(w,x_0)$. 
\end{lemma}
\begin{proof}
To obtain a contradiction assume that there is an external walk $x'\mapright{u}x''\subseteq \Gamma(v,x_0)$ for which there is no external walk $\psi(x')\mapright{u}\psi(x'')$ in $\Gamma(w,x_0)$, or vice versa there is some external walk $y'\mapright{u'}y''$ in $\Gamma(w,x_0)$ for which there is no external walk $\psi^{-1}(y')\mapright{u'}\psi^{-1}(y'')$ in $\Gamma(v,x_0)$. Among such walks choose one of minimal length, and without loss of generality suppose that 
\[
\pi_1=x_1\mapright{u}x_{k}=x_1\mapright{u_1}x_{2}\mapright{u_2}\ldots x_{k-1}\mapright{u_{k-1}}x_{k}
\]
in $\Gamma(v,x_0)$ is external, but there is no external walk with endpoints $\psi(x_1), \psi(x_k)$ and labeled by $u$ in $\Gamma(w,x_0)$. Now, it is straightforward to show that if there is a walk $z\mapright{v}y\subseteq \Delta^C(v,x_0)$ connecting a vertex $z\in \Delta(v,x_0)$ with a vertex $y$ with $\|y\|_{x_0}\ge \|v\|_{x_0}$, then there is a subwalk $z'\mapright{v'}y\subseteq \Gamma(v,x_0)$ for some $z'\in\Delta(v,x_0) $. Thus, from this observation it follows that there are two vertices $x,y\in\Delta(v,x_0)$ for which we may find two paths $x\mapright{w_1}x_1, x_k\mapright{w_2}y\subseteq \Gamma(v,x_0)\cap  \Delta^C(v,x_0)$. Let $x_0\mapright{v_1} x$, $y\mapright{v_2}x_0$ be two geodesics and consider the circuit:
\[
\ell=x_0\mapright{v_1} x\mapright{w_1}x_1\mapright{u}x_{k}\mapright{w_2}y\mapright{v_2}x_0
\]
Now, by applying Lemma~\ref{lem: existence of v-edge} to $\ell$, we conclude that there is a $V$-edge $\hat{e}=\iota(\hat{e})\mapright{X}\tau(\hat{e})$ that gives rise to the following subwalk
\[
\ell(\iota(\hat{e}), \tau(\hat{e}))=\iota(\hat{e})\mapright{v_1'}x\vvlongmapright{w_1uw_2} y\mapright{v_2'}\tau(\hat{e})
\]
with the property $x\mapright{v_1'^{-1}}\iota(\hat{e}), \tau(\hat{e})\mapright{v_2'^{-1}}y\subseteq \Delta^C(v,x_0)$. Thus, the following walk
\[
x\mapright{v_1'^{-1}}\iota(\hat{e})\mapright{X}\tau(\hat{e})\mapright{v_2'^{-1}}y
\]
is contained in $\Delta^C(v,x_0)$. Using the isomorphism $\psi$, we may consider the corresponding image walk in $\Delta^C(w,x_0)$ 
\[
\psi(x)\mapright{v_1'^{-1}}\psi(\iota(\hat{e}))\mapright{X}\psi(\tau(\hat{e}))\mapright{v_2'^{-1}}\psi(y)
\]
with $\psi(x), \psi(y)\in\Delta(w,x_0)$ since $\psi$ is norm-preserving. Since $v_1'w_1uw_2v_2'\in L(X)$, by Lemma~\ref{lem: existence of derived walks} there is the following walk 
\[
\psi(x)\mapright{v_1'^{-1}}\psi(\iota(\hat{e}))\mapright{v_1'w_1}\psi(x_1)\mapright{u} \psi(x_k)\mapright{w_2v_2'}\psi(\tau(\hat{e}))\mapright{v_2'^{-1}}\psi(y)
\]
hence by the determinism of $\Gamma$ we have the walk $\psi(x)\mapright{w_1}\psi(x_1)\mapright{u} \psi(x_k)\mapright{w_2}\psi(y)$.
Now, since the walks $x\mapright{w_1}x_1$, $x_k\mapright{w_2}y$ are in $\Delta^C(v,x_0)$, we deduce that $\psi(x)\mapright{w_1}\psi(x_1)$ and $ \psi(x_k)\mapright{w_2}\psi(y)$ are walks in $\Delta^C(w,x_0)$. Since $x_2\in \Gamma(v,x_0)\setminus\Delta^C(v,x_0)$, we conclude that $\psi(x_2)\in \Gamma(w,x_0)\setminus \Delta^C(w,x_0)$ since otherwise $\psi(x_1)\mapright{u_1}\psi(x_2)$ would belong to $\Delta^C(w,x_0)$, and so also $x_1\mapright{u_1} x_2$ would be an edge in $\Delta^C(v,x_0)$ which would contradict the fact that $\pi_1$ is an external walk. Therefore, there is a maximal subwalk 
\[
\psi(x_1)\vvlongmapright{u_1\ldots u_{j-1}}\psi(x_j)
\]
of $\psi(x_1)\mapright{u}\psi(x_k)$ that is contained in $\Gamma(w,x_0)$ and such that it is external. Since $\psi(x_1)\mapright{u}\psi(u_k)$ is not external, we deduce $j<k$. Thus, by the minimality condition on $\pi_1$, we necessarily have that $x_1\vvlongmapright{u_1\ldots u_{j-1} }x_j$ is an external walk. Hence, $x_j\in \Delta^C(v,x_0)$, $j<k$, however, this contradicts the fact that $\pi_1$ is an external walk.
\end{proof}
\begin{figure}[htp]
\includegraphics[scale=1.3]{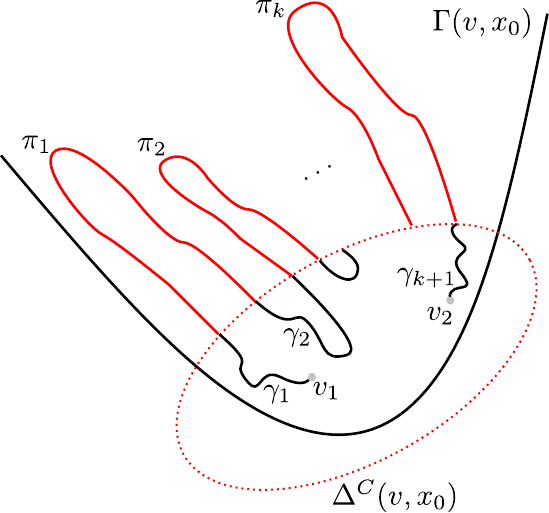}
\caption{In red the external walks $\mathcal{E}(p, C)$ of the walk $p=v_1\mapright{w}v_2$.} \label{fig: external paths}
\end{figure}
We have the following proposition.
\begin{prop}\label{prop: norm-preserving iso are end-isomorphic}
Put $C\ge 2K$. Let $v,w$ be two vertices with $\|v\|_{x_0},\|w\|_{x_0}>K$ such that there is a norm-preserving isomorphism $\psi:  \Delta^C(v,x_0)\to  \Delta^C(w,x_0)$ with $w=\psi(v)$. Then the two inverse subgraphs $\Gamma(v,x_0), \Gamma(\psi(v),x_0)$ are end-isomorphic. 
\end{prop}
\begin{proof}
Before proving the proposition, we first show that the following property holds
\begin{itemize}
\item [P)] if there is a walk $p=v_1\mapright{w}v_2$ in $\Gamma(v,x_0)$ with $v_1,v_2\in \Delta(v,x_0)$, then there is also a walk $\psi(v_1)\mapright{w}\psi(v_2)$ in $\Gamma(\psi(v),x_0)$ with $\psi(v_1),\psi(v_2)\in \Delta(\psi(v),x_0)$. Vice versa, for any walk $z_1\mapright{w}z_2$ in $\Gamma(\psi(v),x_0)$ with $z_1,z_2\in \Delta(\psi(v),x_0)$, there is a walk $\psi^{-1}(z_1)\mapright{w}\psi^{-1}(z_2)$ in $\Gamma(v,x_0)$ with $\psi^{-1}(z_1),\psi^{-1}(z_2)\in \Delta(v,x_0)$.
\end{itemize}
Suppose $\mathcal{E}(p,C)=\{\pi_1,\ldots, \pi_k\}$ and factorize $p$ accordingly as
\begin{equation}\label{eq: walk external}
p=v_1\mapright{w}v_2=\gamma_1\pi_1\gamma_2\pi_2\ldots \gamma_{k}\pi_{k}\gamma_{k+1}
\end{equation} 
where $\gamma_i$ are walks contained in $\Delta^C(v,x_0)$, see again Fig.~\ref{fig: external paths} for a representation of such a factorization. We prove $P)$ by induction on the number $k$ of external walks. If we do not have external walks, then $p=\gamma_1$ is totally contained in $\Delta^C(v,x_0)$, thus $\psi(p)$ is totally contained in $\Delta^C(\psi(v),x_0)$ with $\psi(v_1), \psi(v_2)\in\Delta(\psi(v),x_0)$. In case $p=\gamma_1\pi_1\gamma_2$ has just one external walk $\pi_1=x'\mapright{u}x''$, by Lemma~\ref{lem: external walk} we conclude that $\psi(x')\mapright{u}\psi(x'')$ is also an external walk in $\Gamma(\psi(v),x_0)$. Since $\gamma_1,\gamma_2$ belong to $\Gamma(v,x_0)$ they are formed by vertices $x$ whose norm $\|x\|_{x_0}\ge \|v\|$, hence since $\psi$ is norm-reserving, we deduce that $\psi(\gamma_1), \psi(\gamma_2)$ are walks in $\Gamma(\psi(v),x_0)\cap \Delta^C(\psi(v),x_0)$ with $\iota(\psi(\gamma_1)), \tau(\psi(\gamma_2))\in \Delta(\psi(v),x_0)$. With a slight abuse of notation let us denote the path $\psi(x')\mapright{u}\psi(x'')$ by $\psi(\pi_1)$. By patching this walk together with $\psi(\gamma_1), \psi(\gamma_2)$, we conclude that $\psi(p)=\psi(\gamma_1)\psi(\pi_1)\psi(\gamma_2)$ is a walk in $\Gamma(\psi(v),x_0)$ with endpoints in $\Delta(\psi(v),x_0)$. We have just shown that in case $p$ has at most one external walk, property $P)$ holds. To conclude the proof by induction in case we have more than one external walk we split $p$ into two walks $p_1=\gamma_1\pi_1\delta_1^{-1}$, $p_2=\delta_1\gamma_2\pi_2\ldots \gamma_{k}\pi_{k}\gamma_{k+1}$ where $\delta_1$ is a walk in $\Gamma(v,x_0)\cap \Delta^C(v,x_0)$ with $\iota(\delta_1)\in\Delta(v,x_0)$, $\tau(\delta_1)=\iota(\gamma_2)$, and then apply induction to get that $\psi(p_1)=\psi(\gamma_1)\psi(\pi_1)\psi(\delta_1^{-1})$ and $\psi(p_2)=\psi(\delta_1)\psi(\pi_2\ldots \gamma_{k}\pi_{k}\gamma_{k+1})$ are walks in $\Gamma(\psi(v),x_0)$ with endpoints in $\Delta(\psi(v),x_0)$ and then by using the determinism of $\Gamma(\psi(v),x_0)$ we conclude that also $\psi(p)$ is a walk contained in $\Gamma(\psi(v),x_0)$. By repeating the same argument with $\psi^{-1}$ instead of $\psi$ it is possible to show the second part of property $P)$. 
\\
Let us now construct an end-isomorphism between $\Gamma(v,x_0)$ and $\Gamma(\psi(v),x_0)$. Firstly note that $\psi$ is a bijection between the frontier points $\Delta(v,x_0), \Delta(\psi(v),x_0)$, so we define $\varphi:\Gamma(v,x_0)\to \Gamma(\psi(v),x_0)$ to be equal to $\psi$ on $\Delta(v,x_0)$. Now, for any vertex $y\in \Gamma(v,x_0)$ let $v_1\mapright{u_1}y\mapright{u_2}v_2$ in $\Gamma(v,x_0)$ be a walk with $v_1,v_2\in\Delta(v,x_0)$. By property $P)$ there is walk $\psi(v_1)\mapright{u_1}z\mapright{u_1}\psi(v_2)$ in $\Gamma(\psi(v),x_0)$ with $\psi(v_1), \psi(v_2)\in\Delta(\psi(v),x_0)$. Now, define $\varphi$ by putting $\varphi(y)=z$. This is a well-defined morphism. Indeed, let $v_1'\mapright{s_1}y\mapright{s_2}v_2'$ be another walk for some $v_1', v_2' \in \Delta(v,x_0)$, then using property $P)$ applied to the walk $v_1\mapright{u_1}y\mapright{s_1^{-1}}v_1'$, we conclude that $\psi(v_1)\mapright{u_1}z\mapright{s_1^{-1}}\psi(v_1')$ is also a walk in $\Gamma(\psi(v),x_0)$, thus by the determinism of $\Gamma$ we get that $\varphi(y)$ does not depend on the chosen walk. It is now routine to check that $\varphi: \Gamma(v,x_0)\to\Gamma(\psi(v),x_0)$ is a morphism of inverse graphs that is a bijection by the second part of property $P)$. 
\end{proof}
We have the following characterization that extends and generalizes \cite[Theorem 4.6]{Ce-Wo} to the non-complete case. 
\begin{prop}\label{prop: equivalence context-free}
For an inverse graph $\Gamma$ the following conditions are equivalent:
\begin{enumerate}
\item for any non-empty finite set $F\subseteq V(\Gamma)$ the language $L(\Gamma, x_0,F)$ is context-free;
\item the language $L(\Gamma,x_0)$ is context-free;
\item the rooted graph $(\Gamma, x_0)$ is context-free in the sense of Muller-Schupp;
\item there exists an inverse PDA $M$ with initial state $q_0$ and initial stack symbol $\perp$ such that the rooted configurations graph $(\Gamma(M), (q_0,\perp))$ is isomorphic to $(\Gamma, x_0)$;
\item for any vertex $y_0\in V(\Gamma)$, the rooted graph $(\Gamma, y_0)$ is context-free in the sense of Muller-Schupp.
\end{enumerate}
\end{prop}
\begin{proof}
$(1)\Rightarrow (2)$. It follows by taking $F=\{x_0\}$.
\\
 $(2)\Rightarrow (3)$. For a fixed constant $C\ge 2K$, up to norm-preserving isomorphism, there are finitely many decorated subgraphs $\Delta^C(v,x_0)$ of $\Gamma_{\mathcal{G}}$. Since there a finitely many end-cones $\Gamma(v,x_0)$ with $\|v\|_{x_0}\le K$, then by Proposition~\ref{prop: norm-preserving iso are end-isomorphic} we deduce that there are finitely many end-cones up to end-isomorphism. Hence, the graph $(\Gamma, x_0)$ is context-free in the sense of Muller-Schupp.
 \\
$(3)\Rightarrow (4)$. In the proof of \cite[Lemma 2.3]{muahu2} the authors explain a way to associate a (canonical) pushdown automaton $M$ to the context-free graph $(\Gamma, x_0)$ in such a way that the graph of transitions $\Gamma(M)$ rooted at the initial configuration $(q_0, \perp)$ is isomorphic (as rooted $\wt{A}$-directed graphs) to $(\Gamma, x_0)$. Since $\Gamma$ is inverse we have that also $\Gamma(M)$ is inverse, hence $M$ is an inverse PDA. 
\\
$(4)\Rightarrow (5)$. It follows from \cite[Lemma 2.5]{muahu2}.
\\
$(5)\Rightarrow (3)$. Trivial. 
\\
$(3)\Rightarrow (1)$. Since \cite[Theorem 4.2]{Ce-Wo} holds for any connected involutive graph that is context-free in the sense of Muller-Schupp, we have that for any $y\in F$ the language $L(\Gamma, x_0,y)$ is context-free. Since the (finite) union of context-free languages is also context-free \cite{hop}, we conclude that
\[
L(\Gamma, x_0, F)=\bigcup_{y\in F}L(\Gamma, x_0,y)
\]
is also context-free.
\end{proof}
This characterization settles the name context-free for an inverse graph since such a graph can be either seen (rooted at any vertex) as an acceptor of a context-free language or as a configurations graph of an inverse PDA or, more geometrically, being context-free in the sense of Muller-Schupp.


\section{Context-free inverse graphs are quasi-isometric to a tree}\label{sec: geometrical considerations}

We now show that context-free inverse graphs have a tree-like shape in the sense that they are quasi-isometric to a tree. We recall that a quasi-isometry between the two metric spaces $(X_1,d_1), (X_2,d_2)$ is a map $\psi:X_1\to X_2$ for which there are positive constants $k,K,c,C$ such that for all $x, y\in X_1$ the following inequalities 
\[
kd_1(x,y)-c\le d_2(\psi(x),\psi(y))\le K d_1(x,y)+C
\]
hold, and $\psi$ is almost surjective, in the sense that there exists a constant $N$ such that for $z\in X_2$ there is some $x\in X_1$ for which $d(\psi(x),z)\le N$ holds. The idea of tree-like may be reformulated using several other notions, see for instance \cite[Proposition 3.1]{Pedro Nora}, \cite[Theorem 4.7]{Antolin}. We briefly recall some of them and we add another equivalence that will be useful later. 
\\
A geodesic polygon $P$ of $\Gamma$ is a sequence $P=p_1\ldots p_n$, $n\ge 2$, of geodesics walks with $\iota(p_{i+1})=\tau(p_i)$ for $1\le i\le n-1$, and $\iota(p_1)=\tau(p_n)$. Each walk $p_i$ forms one side of $P$. Let $\delta>0$ be a constant, we say that $P$ is $\delta$-thin if for all $i\in [1,n]$, and for all vertices $v\in p_i$, there exists a vertex $w\in p_j$ for some $j\neq i$ such that $d(w,v)\le \delta$. In case all geodesic polygons of $\Gamma$ are $\delta$-thin, we say that $\Gamma$ is polygon $\delta$-hyperbolic, or simply polygon hyperbolic. This notion is a special case of Rips condition for $\delta$-hyperbolic spaces. 
\\
For a graph $\Gamma$ and a partition $\mathcal{P}$ of the vertex set $V(\Gamma)$, we may define the undirected graph $\Gamma/\mathcal{P}$ as the simple graph having as set of vertices $\mathcal{P}$ and two different subsets $S_1,S_2\in \mathcal{P}$ are connected by an edge if there are $u_1\in S_1, u_2\in S_2$ and an edge $u_1\mapright{} u_2\in \Gamma$. In case there is some integer $N$ such that $\delta_{\Gamma}(S)\le N$ for all $S\in \mathcal{P}$, we say that $\mathcal{P}$ has \emph{uniform diameter}. A partition $\mathcal{P}$ such that $\Gamma/\mathcal{P}$ is a tree, is called a \emph{strong tree decomposition}. 
The \emph{cone} of a vertex $v$ with respect to $x_0$ is the induced subgraph obtained by all the vertices $x$ that are in any geodesic connecting $x_0$ with $x$ and containing $v$, i.e., 
\[
C(v,x_0)=\left\la x\in V(\Gamma): \|x\|_{x_0}=\|v\|_{x_0}+d(v,x)\right\ra_{\Gamma}
\]
A central notion for simple graphs is that of \emph{tree decomposition}, see for instance \cite{Diestel}. For us a tree decomposition of $\Gamma$ is a pair  $(T,\mathcal{V})$ where $T$ is a simple graph that is a tree, and $\mathcal{V}=\{V_t\}_{t\in T}$ is a collection of subsets $V_t\subseteq V(\Gamma)$ indexed by the vertices of $T$ satisfying the following properties:
\begin{itemize}
\item[T1)]  $\bigcup_{t\in T}V_t=V(\Gamma)$;
\item[T2)] for every edge $e\in E(\Gamma)$, there is some $t\in T$ such that $\iota(e), \tau(e)\in V_t$;
\item[T3)] if $V_t$ and $V_s$ both contain a vertex $v$, then all nodes $V_r$ of the tree in the (unique) path between $V_t$ and $V_s$ contain $v$ as well.
\end{itemize}
If there is some uniform constant $N\ge 1$ such that $\delta_{\Gamma}(V_t)\le N$ for all $t\in T$, we say that $(T,\mathcal{V})$ is a tree decomposition with \emph{uniform diameter}. Since $\Gamma$ is an inverse $\wt{A}$-graph with $A$ finite, the out-degree of each vertex is finite, so any subgraph with finite diameter is finite, thus the condition that $(T,\mathcal{V})$ has uniform diameter, implies that each subset $V_t$ is uniformly bounded by some constant $C$. Hence, in the notation of  \cite{Armin}, $(T,\mathcal{V})$ is a tree decomposition of treewidth $C$.
The following proposition characterizes tree-like inverse graphs in terms of all the previous definitions. 
\begin{prop}\label{prop: charact tree-like}
For a connected inverse graph $\Gamma$, the following conditions are equivalent:
\begin{enumerate}
\item $\Gamma$ is quasi-isometric to a tree;
\item $\Gamma$ is polygon hyperbolic;
\item there is a strong tree decomposition $\mathcal{P}$ of uniform diameter; 
\item there exists a constant $\delta\ge 1$ such that if $x_0\in V(\Gamma)\setminus D_{\delta}(v)$ and $x\in C(v,x_0)$, then there is no walk in $\Gamma\setminus D_{\delta}(v)$ connecting $x_0$ to $x$;
\item $\Gamma$ has a tree decomposition with uniform diameter;
\end{enumerate}
\end{prop}
\begin{proof}
Conditions $(1)-(4)$ are equivalent by \cite[Proposition 3.1]{Pedro Nora}. \\
$(3)\Rightarrow (5)$. Take $T=\Gamma/\mathcal{P}$ and for $t\in T$, let $S_t$ be the corresponding subset of vertices. Clearly property T1) is already satisfied, however, to satisfy T2) it is enough to add to each $S_t\in\mathcal{P}$ the adjacent vertices, so we augment each $S_t$ by considering 
\[
S'_t=S_t\cup \{v: \mbox{ there is an edge }e\in E(\Gamma)\mbox{ with }v=\tau(e),\mbox{ and } \iota(e)\in S_t\}
\]
in this way properties T1) and T2) are both satisfied. Furthermore, for any two adjacent vertices $t_1, t_2\in T$, the set $S'_{t_1}\cap S'_{t_2}$ is formed by pairs of vertices $v_1,v_2$ such that there is an edge $v_1\mapright{}v_2$ connecting $v_1\in S_{t_1}$ with $v_2\in S_{t_2}$. Hence, the only vertices $r\in T$ such that $S'_{t_1}\cap S'_{t_2}\subseteq S'_r$ are $r=t_1,t_2$, so T3) is satisfied, and this concludes the proof that $(T,\{S'_t\}_{t\in T})$ is a tree decomposition. Now, it is straightforward to verify that $\delta_{\Gamma}(S_t)\le N$ for all $S\in \mathcal{P}$ implies that $\delta_{\Gamma}(S_t)\le N+2$, hence $(T,\{S'_t\}_{t\in T})$ is a tree decomposition with uniform diameter. 
\\
$(5)\Rightarrow (2)$. Let $(T,\{V_t\}_{t\in T})$ be a tree decomposition with uniform diameter $N$ and let $P=p_1,\ldots, p_n$ be a geodesic polygon of $\Gamma$. Let us show that $P$ is $2N$-thin, so let  $v\in p_i$ be a fixed vertex, and let us show that the distance of $v$ from the rest of the other walks is at most $2N$. For each vertex $v$ of the walk $p_i$, let $T(v)=\{t\in T: v\in V_t\}$. By condition T3),  $T(v)$ has the following closure property: for any $x,y\in T(v)$ all the vertices in $T$ in the geodesic connecting $x$ to $y$ are contained in $T(v)$, hence $T(v)$ may be visualized as a subtree of $T$. Now, note that if $v,w$ are two adjacent vertices of $\Gamma$, $T(v)\cap T(w)\neq \emptyset$, thus the walk $p_i$ induces a subtree $T(p_i)$ of $T$ obtained by taking the union of $T(v)$ for each $v\in p_i$. Let $\pi=p_{i+1}\ldots p_np_1\ldots p_{i-1}$ be the walk obtained by composing all the walks of $P$ except $p_i$. Since $\pi$ and $p_i$ have the same endpoints $x=\iota(p_i)=\tau(\pi)$, $y=\iota(\pi)=\tau(p_i)$, then $T(\pi)$ and $T(p_i)$ are two subtrees of $T$ containing both $T(x), T(y)$. It is not difficult to show that we may choose a walk $q=t_1\ldots t_s\ldots t_k$ in $T(p_i)$ that is compatible with $p_i$ in the following sense: there is a map $f$ from the vertices of $q$ to the vertices of $p_i$ such that $f(t_1)=x\in V_{t_1}$, $f(t_k)=y\in V_{t_k}$, and if $t_s, t_{s+1}$ are two consecutive vertices of $q$
with $f(t_s)\in V_{t_s}$, then $f(t_{s+1})\in V_{t_{s+1}}$ and either $f(t_{s+1})=f(t_s)$, or there is an edge $f(t_{s+1})\mapright{}f(t_s)$ belonging to $p_i$. Note that $f$ is surjective, and let $t_j\in q$ be a vertex with $f(t_j)=v\in V_{t_j}$. Since $T$ is a tree, $q$ contains (as a subgraph) the geodesic $q'$ connecting $t_1$ with $t_k$ and this is also contained in $T(\pi)$ since it contains also the vertices $t_1,t_k$. We have two cases: either $t_j$ belongs to $q'$, or it does not. In the first case since $q'$ is contained in both $T(p_i)$ and $T(\pi)$ we deduce that there is a vertex $w\in \pi$ such that $w\in V_{t_j}$, and since $v$ is also contained in $V_{t_j}$ and $\delta_{\Gamma}(V_{t_j})\le N$ we deduce that $d(v,w)\le N$, so the distance of $v\in p_i$ from $\pi$ is at most $N$. In case $t_j$ is not contained in the geodesic $q'$, since $T$ is a tree, $q$ must contain a circuit $t'\ldots  t_j\ldots t'$ where $t'$ belongs to the geodesic $q'$ and the image of the circuit $t'\ldots  t_j\ldots t'$ through $f$ is a subwalk $v_1\mapright{}\ldots v\mapright{}\ldots v_2 $ of $p_i$ with $v_1,v_2\in V_{t'}$. Since $p_i$ is a geodesic, then also the subwalk $v_1\mapright{}\ldots v\mapright{}\ldots v_2 $ is a geodesic, hence $d(v_1,v)\le d(v_1,v_2)\le N$ since $v_1,v_2\in V_{t'}$. Since $t'$ is in $q'$ by the same argument as before, we may find $w\in \pi$ with $w\in V_{t'}$. Hence, we deduce $d(v_1,w)\le N$ which together with $d(v_1,v)\le N$ yields $d(v,w)\le 2N$ and this concludes the proof. 
\end{proof}
Context-free inverse graphs share all the properties stated in Proposition~\ref{prop: charact tree-like}, this is contained in the following proposition. 
\begin{prop}\label{prop: context-free implies tree-like}
Let $(\Gamma,x_0)$ be a rooted inverse graph that is context-free, then $\Gamma$ is quasi-isometric to a tree.
\end{prop}
\begin{proof}
We show that condition $(4)$ of Proposition~\ref{prop: charact tree-like} holds by taking $\delta=3K$. Now, note that $C(v,x_0)$ is a subgraph of $\Gamma(v,x_0)$ since by definition there is a geodesic
\[
x_0=y_1\mapright{}y_2\mapright{}\ldots\mapright{}y_i\mapright{}\ldots\mapright{}y_k=x
\]
with $y_i=v$, for some $i$, moreover all the vertices $y_{j}$ with $j\ge i$ have the property $\|y_j\|_{x_0}\ge \|y_i\|_{x_0}= \|v\|_{x_0}$, and so $x$ is connected to $v$ by a walk contained in $\Gamma(v,x_0)$, i.e., $x\in \Gamma(v,x_0)$. Therefore, any walk connecting $x_0$ to some vertex $x\in C(v,x_0)$ passes from a frontier point in $\Delta(v,x_0)$. By Proposition~\ref{prop: frontier uniformly bounded} $\Delta(v,x_0)\subseteq V(D_{\delta}(v))$, which implies that there is no walk in $\Gamma\setminus D_{\delta}(v)$ connecting $x_0$ to $x$.
\end{proof}

\subsection{A detour to inverse monoids and Sch\"utzenberger graphs}
Sch\"{u}tzenberger graphs are the natural generalization of Cayley graphs in the area of inverse semigroup theory. We know that for Cayley graphs being context-free and tree-like are equivalent properties (since they are transitive). Proposition~\ref{prop: equivalence context-free} in combination with \cite[Theorem 4.13]{Pedro Nora} shows that the same property holds in the finitely presented case, answering to the open problem \cite[Question 6.2]{Pedro Nora}. 
\begin{theorem}\label{theo: schutz with context-free are tree-like}
Let $M = Inv \la A\mid \mathcal{R}\ra$ be a finitely presented inverse monoid, and let $w\in\wt{A}^*$ be a word. Then, the Sch\"utzenberger graph of $w$ is tree-like if and only if the language accepted by the Sch\"utzenberger automaton of $w$ is context-free. 
\end{theorem}
\begin{proof}
We refer the reader to \cite{Pedro Nora, Steph} for more details on Sch\"utzenberger graphs, for our purpose it is enough to know that a Sch\"utzenberger automata $\mathcal{A}(A,\mathcal{R}, w)$ with respect to the presentation $\la A\mid \mathcal{R}\ra$ and the word $w\in \wt{A}$ is a birooted inverse graph $(\Gamma, x_0, \{y_0\})$. One direction of the statement is  \cite[Theorem 4.13]{Pedro Nora}. For the other direction, if the accepted language $L(\Gamma, x_0, \{y_0\})$ is context-free, then by Proposition~\ref{prop: equivalence context-free} $\Gamma$ is context-free, and so by Proposition~\ref{prop: context-free implies tree-like} it is also quasi-isometric to a tree.
\end{proof}
Since for groups the previous equivalence holds for any finitely generated group, the natural question arises as to whether we may extend the previous theorem to a generic finitely generated inverse monoid. We now show that this is not the case, and the counterexample is contained in \cite[Example 4.4]{Pedro Nora}. We first need the following fact.
\begin{prop}
Let $(\Gamma,x_0)$ be a rooted inverse graph that is context-free. Then the language of the geodesics rooted at a vertex $x_0$, that is the set
\[
\Geo(\Gamma, x_0)=\{w\in\wt{A}^*: x_0\mapright{w}x\mbox{ is a geodesic}\}
\]
is a regular language. 
\end{prop}
\begin{proof}
Following the same strategy of the proof of \cite[Theorem 4.5]{Pedro Nora}, it is enough to show that there are finitely many cones $C(v,x_0)$ up to end-isomorphism, which in this context means an isomorphism $\psi:C(v,x_0)\to C(w,x_0)$ with $\psi(v)=w$. Now, by Proposition~\ref{prop: equivalence context-free} $\Gamma$ is a context-free graph in the sense of Muller-Schupp, so there are finitely many end-cones graphs $\Gamma(v,x_0)$ up to end-isomorphism (of end-cones). We have already seen in the proof of Proposition~\ref{prop: context-free implies tree-like} that each cone $C(v,x_0)$ is contained in $\Gamma(v,x_0)$. Let \[\psi:\Gamma(v_1,x_0)\to \Gamma(v_2,x_0)\] be an end-isomorphism. We claim that if $x_0\mapright{w_1}v_1\mapright{w_2}u$ is a geodesic with $u\in C(v_1,x_0)$, then for any geodesic $x_0\mapright{s_1}\psi(v_1)$, also the walk $x_0\mapright{s_1}\psi(v_1)\mapright{w_2}\psi(u)$ is a geodesic. Suppose by contradiction that there is a shorter walk $x_0\mapright{s_1'}z\mapright{w_2'}\psi(u)$ where $z\in\Delta(v_2, x_0)$, since $\|z\|_{x_0}=\|\psi(v_1)\|=\|v_2\|_{x_0}$ we deduce that the length of the walk $z\mapright{w_2'}\psi(u)$ is shorter than $\psi(v_1)\mapright{w_2}\psi(u)$. Thus, the walk $\psi^{-1}(z)\mapright{w_2'}u$ is shorter than $v_1\mapright{w_2} u$, and this contradicts the fact that $x_0\mapright{w_1}v_1\mapright{w_2}u$ is a geodesic. 
Hence, $\psi(u)\in C(\psi(v_1),x_0)$ and $\psi:C(v_1, x_0)\to C(\psi(v_1), x_0)$ is an end-isomorphism of cones. Since a cone $C(v_1,x_0)$ is completely determined by the vertex $v_1$ and the graph $\Gamma(v_1,x_0)$, and there are finitely many end-isomorphic classes of graphs $\Gamma(v,x_0)$, and $\Delta(v,x_0)$ is uniformly upper bounded by Proposition~\ref{prop: frontier uniformly bounded}, then we deduce that there are also finitely many cones up to end-isomorphism. 
\end{proof}
In \cite[Example 4.4]{Pedro Nora} it is described an inverse monoid on $A=\{a,b,c\}$ that is not finitely presented, whose Sch\"utzenberger graph $\Gamma$ of the element $aa^{-1}$ is a tree, but the language of the geodesics rooted at $aa^{-1}$ is not a regular language, see Fig.\ref{fig: bicyclic}. Hence, by the previous proposition, we conclude that $\Gamma$ can not be a context-free graph. Therefore, if we remove the condition of being finitely presented the ``only if'' condition of the previous theorem is no longer valid.
\begin{figure}[htp]
\includegraphics[scale=1.3]{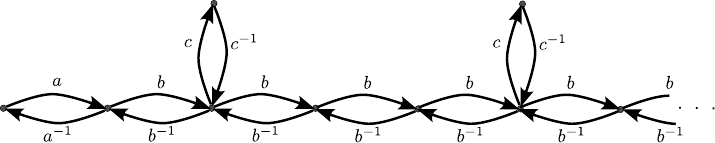}
\caption{The Sch\"utzenberger graph of $aa^{-1}$ of the inverse monoid with presentation $Inv\la a,b,c\mid aa^{-1}b^{n^2}cc^{-1}b^{-n^2}=aa^{-1}, n\ge 1 \ra$.} \label{fig: bicyclic}
\end{figure}


\section{Characterization of quasi-transitive context-free inverse graphs}\label{sec: main result}
In this section, we shall prove the characterization of context-free inverse graphs that are quasi-transitive, but first, we need to make some considerations on the language $L(\Gamma,x_0)$ of a rooted inverse graph that will be used in the proof. 
\\
For each vertex $v\in V(\Gamma)$ we denote by $\mathcal{D}(\Gamma, v)$ the set of symmetric Dyck words $u\in\wt{A}$ labeling a reduced circuit $v\mapright{u}v$ starting at $v$, i.e., a word $u\in\wt{A}$ with $\oo{u}=1$ in the free group $\mathbb{F}_A$. Clearly if a word $w$ factorizes as $w=u_1v_1u_2v_2\ldots u_{k-1}v_{k-1}u_{k}$ with $\oo{u_i}=1$ and $w$ labels the walk $x_0\mapright{w}z$ then $x_0=y_1\mapright{v_1}y_2\mapright{v_2}\ldots y_{k-1}\mapright{v_{k-1}}y_k=z$ is also a walk with $u_i\in \mathcal{D}(\Gamma, y_i)$. In this case we say that $w$ is \emph{derivable} from $v_1\ldots v_{k-1}$ in $\Gamma$. In particular, by Lemma~\ref{lem: there are reduced walks} any word $w$ labeling a walk in $\Gamma$ is derivable from $\oo{w}$. The following lemma shows that the notion of derivability transfers to the image graphs in case there is a covering. 
\begin{lemma}\label{lem: expansion set}
Let $\psi:  \Gamma\to \Lambda$ be a covering of inverse graphs. Then, $\mathcal{D}(\Gamma, v)=\mathcal{D}(\Lambda, \psi(v))$ for every $v\in V(\Gamma)$. Moreover, the word $w$ is derivable from $v_1\ldots v_{k-1}$ in $\Gamma$ if and only $w$ is derivable from $v_1\ldots v_{k-1}$ in $\Lambda$.
\end{lemma}
\begin{proof}
The inclusion $\mathcal{D}(\Gamma, v)\subseteq \mathcal{D}(\Lambda, \psi(v))$ follows by the fact that both $\Gamma$ and $\Lambda$ are inverse and $\psi$ is a morphism of inverse graphs. As for the other inclusion, by the lifting property of Proposition~\ref{prop: lifting of coverings}, any circuit $\psi(v)\mapright{u}\psi(v)$ with $u\in \mathcal{D}(\Lambda, \psi(v))$ lifts to the circuit $v\mapright{u} v$ since $\Gamma$ is inverse and $\oo{u}=1$, i.e., $u\in \mathcal{D}(\Gamma, v)$. The last statement follows from the previous equality. 
\end{proof}

\begin{lemma}\label{lem: closure property for coverings}
Let $\psi:  \Gamma\to \Lambda$ be a covering of inverse graphs, where $\Lambda$ is finite. Suppose that $L$ is a context-free language with $\oo{L(\Gamma, x_0)}\subseteq L\subseteq L(\Gamma, x_0)$. Then, $L(\Gamma,x_0)$ is context-free.
\end{lemma}
\begin{proof}
 Since $\oo{L(\Gamma, x_0)}\subseteq L$ that is accepted by some pushdown automaton $M$, it is enough to show that there is a pushdown automaton $\mathcal{M}$ accepting all and only all the derivable words of $L$ in $\Gamma$ that is, by Lemma~\ref{lem: expansion set}, equivalent to show that $\mathcal{M}$ accepts all the derivable words of $L$ in $\Lambda$.  We first show that there is a pushdown automaton $\mathcal{D}_y$ that accepts $\mathcal{D}(\Lambda, y)$ for each $y\in V(\Lambda)$, and then we show that such pushdown automata can be merged with $M$ to obtain $\mathcal{M}$. The set of states of $\mathcal{D}_y$ coincides with $V(\Lambda)$, the initial and final state is $y$ and the alphabet of the stack is $\wt{A}\cup\{\perp_y\}$. The transitions are defined as follows: 
\begin{align*}
&(x_1, b, a) \sststile{\mathcal{D}_y}{}(x_2, ab),\,\mbox{ if }b\neq a^{-1}\mbox{ and }x_1\mapright{b}x_2\mbox{ is an edge in }\Lambda\\
&(x_1, b, a) \sststile{\mathcal{D}_y}{}(x_2, 1),\,\mbox{ if }b= a^{-1}\mbox{ and }x_1\mapright{b}x_2\mbox{ is an edge in }\Lambda
\end{align*}
Now, starting at the state $y$ with the stack containing the initial symbol $\perp_y$, by a simple induction on the length of the input word $u$, it is straightforward to show that by reading $u$ the machine ends in the state $z$ for which we have a walk $y\mapright{u}z$ in $\Lambda$ and on the stack we read the word $\oo{u}$. Therefore, $u\in \mathcal{D}(\Lambda, y)$ if and only if $z=y$ and on the stack we read $\perp_y$.
Let $M=(Q,\wt{A},\Theta, \delta',p_0,\perp,F')$
be the pushdown automaton accepting the language $L$, we first modify it by considering the automaton $M'=(Q\times V(\Lambda),\wt{A},\Theta, \delta, (p_0,\psi(x_0)), \perp, F)$ that has a transition $((p_1,x_1), a, \beta) \sststile{M'}{}((p_2,x_2), \gamma)$ whenever $(p_1, a, \beta) \sststile{M}{}(p_2, \gamma)$ is a transition of $M$, and $x_1\mapright{a}x_2$ is an edge in $\Lambda$ (possibly empty). The set of final states is given by $F=\{(q,x): q\in F', x\in V(\Lambda)\}$. In this way, if we take as the initial state of $M'$ the pair $(p_0, \psi(x_0))$, where $p_0$ is the initial state of $M$, then $M'$ accepts the same language as $M$, but it has the extra advantage that when $M'$ reads the input $u$ it reaches a state $(p',y)$ where the second component keeps track of the vertex that we have reached after reading $u$ starting at $\psi(x_0)$ in $\Lambda$. We now glue the automata $\mathcal{D}_y$, $y\in V(\Lambda)$, to $M'$ to obtain the automaton $\mathcal{M}$. More precisely, let us number the set of states of $M'$ as $t_1, \ldots, t_k$. Now consider the pushdown automaton
\[
\mathcal{M}=(Q',\wt{A},\Theta\cup  \wt{A}\cup \{\perp^{(i)}_y: y\in V(\Lambda), i\in[1,k]\}, \delta'', (p_0,\psi(x_0)), \perp, F)
\]
 obtained from $M'$ by adding for each state $t_i=(p_1, y)$ of $M'$ a copy of the automaton $\mathcal{D}_y$ that we denote by $\mathcal{D}^{(i)}_y$ (states are denoted using the superscript $(i)$ like $x^{(i)}$). The gluing is obtained by adding the following transitions: $((p_1,y), 1, \beta) \sststile{\mathcal{M}}{}(y^{(i)}, \beta\perp_y^{(i)})$, for any $\beta\in\Theta$ and where $\perp^{(i)}_y$ is a new stack symbol, and we add the reverse transition $(y^{(i)}, 1, \perp^{(i)}_y) \sststile{\mathcal{M}}{}((p_1,y), 1)$. Now, if $\mathcal{M}$ starts in the initial state $(p_0, \psi(x_0))$ and by reading $u$ it ends in a state $(p_1, y)$, then $\mathcal{M}$ can nondeterministically choose to take the transition $((p_1,y), 1, \beta) \sststile{\mathcal{M}}{}(y^{(i)}, \beta\perp^{(i)}_y)$ and ``enter'' in the automaton $\mathcal{D}^{(i)}_y$. Suppose that the machine performs such transition, and it is in the state $y^{(i)}$ with the stack containing the word $\alpha\perp^{(i)}_y$, for some word $\alpha$ on the stack alphabet of $\mathcal{M}$. Since the final states are not in the state set of $\mathcal{D}^{(i)}_y$, the machine $\mathcal{M}$ to accept, has to exit from $\mathcal{D}^{(i)}_y$. Suppose that $u'$ is the word that $\mathcal{M}$ has read while performing transitions belonging to $\mathcal{D}^{(i)}_y$. We claim that $u'\in \mathcal{D}(\Lambda, y)$. Indeed, since the only transition to exit, is $(y^{(i)}, 1, \perp^{(i)}_y) \sststile{\mathcal{M}}{}((p_1,y), 1)$, the machine exits only if it is in state $y^{(i)}$ and with the stack containing the word $\alpha\perp^{(i)}_y$. This means that the machine $\mathcal{D}_y$, by reading $u'$ starting at $y$, returns to the same state and with the stack containing only the initial symbol $\perp_y$, hence $u'\in \mathcal{D}(\Lambda, y)$. Thus, by induction on the length of a word $w=a_1\ldots a_k$ it is possible to prove the following claims
\begin{itemize}
\item  if in $M'$ we have the computation $((p_0,\psi(x_0)),\perp)\sdtstile{M'}{u}((q,y),\alpha)$, then for any word $w$ that is derivable from $u$ in $\Lambda$, in $\mathcal{M}$ we have the following computation $((p_0,\psi(x_0)),\perp~)\sdtstile{\mathcal{M}}{w}((q,y),\alpha)$;
\item if in $\mathcal{M}$ we have the computation $((p_0,\psi(x_0)),\perp)\sdtstile{\mathcal{M}}{w}((q,y),\alpha)$, then there is a word $u$ such that $w$ is derivable from $u$ in $\Lambda$, and there is a computation $((p_0,\psi(x_0)),\perp)\sdtstile{M'}{u}((q,y),\alpha)$ in $M'$.
\end{itemize}
In particular, the first claim shows that if $u\in L(M')=L$ and $w$ is derivable from $u$ in $\Lambda$, then $w\in L(\mathcal{M})$, while the second one shows that all the words accepted by $\mathcal{M}$ are derivable from some word in $L(M')=L$.
\end{proof}
We are now in position to prove our characterization. 
\begin{theorem}\label{theo: main theorem}
Let $\Gamma$ be a quasi-transitive infinite inverse graph. T.F.A.E.
\begin{enumerate}
\item $\Gamma$ is context-free;
\item $\Gamma$ is quasi-isometric to a tree;
\item $\Aut(\Gamma)$ is the fundamental group of a finite graph of finite groups;
\item $\Aut(\Gamma)$ is a finitely generated virtually free group;
\item there is a cover $\psi: \Gamma\to \Lambda$ of inverse graphs with $\Lambda$ finite, and a surjective morphism
 $$\oo{\eta}:\pi_1(\Lambda, \psi(x_0))\to \mathbb{F}_X$$ onto a free group $\mathbb{F}_X$ of finite rank, with $\oo{L(\Gamma,x_0)}=\ker(\oo{\eta})$.
\end{enumerate}
\end{theorem}
\begin{proof}
Implication $(1)\Rightarrow (2)$ follows from Proposition~\ref{prop: context-free implies tree-like}.
\\
$(2)\Rightarrow (3)$.  The group $\Aut(\Gamma)$ acts on $\Gamma$ with finitely many orbits by hypothesis and each vertex stabilizer $\Aut(\Gamma)_v=\{\varphi\in \Aut(\Gamma): \varphi(v)=v\}$ is trivial by Lemma~\ref{lem: uniqueness of automorphism}. Since by Proposition~\ref{prop: charact tree-like} $\Gamma$ has a tree decomposition with uniform diameter, it has finite treewidth (in particular as an unlabeled graph). By a result of Diekert and Weiß \cite[Corollary 5.10]{Armin} we conclude that $\Aut(\Gamma)$ is the fundamental group of a finite graph of finite groups.
\\
$(3)\Rightarrow (4)$. It is a well-known fact, see for instance \cite[Proposition 11 page. 120]{Serre}.
\\
$(4)\Rightarrow (5)$. Since $\Aut(\Gamma)$ is a finitely generated virtually free group, there is a finite index subgroup $H\simeq \mathbb{F}_X$ of $\Aut(\Gamma)$ that is free of finite rank. Using the same notation introduced in Section~\ref{sec: coverings of inverse graphs}, consider $W=H\cdot x_0$ and the associated relation $\rho_W$. Now, take $\Lambda=\Gamma/\rho_W$, by Lemma~\ref{lem: finite index}, $\Lambda$ is a finite inverse graph and the natural projection $\psi:\Gamma\to \Gamma/\rho_W$ is a covering. By Proposition~\ref{prop: language inverse homomorphism} there is a surjective morphism $\oo{\eta}:\pi_1(\Lambda, \psi(x_0))\to H$ with $\oo{L(\Gamma,x_0)}=\oo{\eta}^{-1}(\mathds{1}_H)$. 
\\
$(5)\Rightarrow (1)$. Put $\mathbb{F}_Y=\pi_1(\Lambda, \psi(x_0))$, and let $\theta_Y:\wt{Y}^*\to \mathbb{F}_Y$, $\theta_X:\wt{Y}^*\to \mathbb{F}_X$ be the natural maps. Now, consider the morphism $\varphi= \oo{\eta}\circ \theta_Y$, by the remark \cite[Lemma~2.1]{herb} there is a morphism $h:\wt{Y}^*\to \wt{X}^*$ such that the following diagram commutes:
\[
\begin{tikzcd}
 \mathbb{F}_Y  \arrow[r, "\oo{\eta}"] & \mathbb{F}_X \\
 \wt{Y}^* \arrow[r, "h"] \arrow["\theta_Y", u, twoheadrightarrow] \arrow[ur, "\varphi" ]& \wt{X}^* \arrow[u, "\theta_X", twoheadrightarrow]
\end{tikzcd}
\]
Thus, $\ker(\oo{\eta})= \oo{H}$ is equal to the set of reduced words of the preimage $H=\varphi^{-1}(1)$. Since the previous diagram commutes, we deduce that $H=h^{-1}(\theta_X^{-1}(1))$, where $\theta_X^{-1}(1)$ is the language of symmetric Dyck words, which is a well known context-free language. Now, since context-free languages are closed under inverse morphisms, see \cite[Theorem 6.3]{hop}, we conclude that $H$ is also context-free. Therefore, $\oo{L(\Gamma,x_0)}=\ker(\oo{\eta})=\oo{H}=H\cap \oo{\wt{Y}^*}$ is also context-free since the set of reduced words $\oo{\wt{Y}^*}$ is regular, and intersection of a context-free language with a regular one is still context-free,  \cite[Theorem 6.5]{hop}. Now, since $\oo{L(\Gamma,x_0)}\subseteq L(\Gamma, x_0)$ holds, by Lemma~\ref{lem: there are reduced walks}, then by Lemma~\ref{lem: closure property for coverings} we conclude that $L(\Gamma, x_0)$ is context-free, i.e., $\Gamma$ is context-free. 
\end{proof}
\begin{rem}
Note that to prove equivalences $(2)\Leftrightarrow (3)\Leftrightarrow (4)$ we could have used the  \v{V}arc-Milnor Lemma \cite[Theorem 23]{harpe} and Muller and Schupp' result. Since the graph $\Gamma$ is locally finite and connected, it is a proper geodesic space, and $\Aut(\Gamma)$ acts by isometries on $\Gamma$. Since this action is proper and the quotient graph $\Aut(\Gamma)\setminus\Gamma$ is finite, by the \v{V}arc-Milnor Lemma $\Aut(\Gamma)$ is finitely generated, and the Cayley graph $\Cay(\Aut(\Gamma), X)$ of $\Aut(\Gamma)$ is quasi-isometric to $\Gamma$ via the mapping $\varphi\mapsto \varphi(x_0)$. Thus, since $\Cay(\Aut(\Gamma), X)$ and $\Gamma$ are quasi-isometric, by Muller and Schupp' result we have that $\Gamma$ is a quasi-tree if and only if $\Aut(\Gamma)$ is virtually-free. 
\end{rem}

The last condition may be seen as a more group theoretic restatement of the Chomsky-Sch\"utzenberger representation theorem for context-free languages \cite{cho,Kambites}, and it is central in proving the characterization of Section~\ref{sec: weaker tara's conjecture}. 
The previous theorem does not hold if we remove the hypothesis of being quasi-transitive. Indeed, consider the inverse graph $\Gamma$ in Fig.~\ref{fig: bicyclic}. There are no non-trivial automorphisms, so it is far from being quasi-transitive. Moreover, $\Gamma$ is tree-like (actually a tree!), but as we have seen at the end of Section~\ref{sec: geometrical considerations} it is not a context-free graph. A natural question concerns the existence of examples of tree-like quasi-transitive inverse graphs that are not in general Cayley graphs. One way to construct such examples consists in considering Sch\"{u}tzenberger graphs of amalgams of finite inverse semigroups. Indeed, take an amalgam $[S_1, S_2; U]$ of finite inverse semigroups $S_1, S_2$ with common inverse semigroup $U$, and consider the amalgamated free product $S_1*_U S_2$. Now, the underlying graph $\Gamma$ of any Sch\"{u}tzenberger automaton of any element $g\in S_1*_U S_2$ is an inverse graph whose automorphism group $\Aut(\Gamma)$ is isomorphic to the maximal subgroup of $S_1*_U S_2$ that contains the idempotent $gg^{-1}$, see \cite{Steph}. In \cite[Theorem 7]{maximal} it is proved that in case the idempotent $gg^{-1}$ is $\mathcal{D}$-related to an idempotent of $U$, then $\Gamma\setminus \Aut(\Gamma)$ is finite and $\Aut(\Gamma)$ is isomorphic to the fundamental groups of finite graphs of finite groups. In particular, thanks to the previous theorem, this shows that such graphs are context-free. This offers an alternative proof to the more constructive and direct one contained in \cite{amalgams}.

\section{A characterization of groups that are virtually finitely generated subgroups of the direct product of free groups}\label{sec: weaker tara's conjecture}
Differently from the regular case, the intersection of context-free languages is not, in general, a context-free language \cite{hop}.  For this reason, it is natural to consider the closure of context-free languages under intersection. The intersection of $k$ context-free languages is called a $k$-context-free language. A language is poly-context-free if it is $k$-context-free for some $k\in\mathbb{N}$ \cite{Tara}.  A group is called $k$-context-free if its word problem is a $k$-context-free language.  A group whose word problem is a poly-context-free language is called a poly-context-free group. Such groups have been introduced in \cite{Tara} as a natural generalization of virtually free groups, and they have been further explored in \cite{TCS}, where it is proved that they have a multipass word problem. Since the class of 1-context-free groups coincides with the class of virtually free groups, it is a natural problem to generalize such result for a generic poly-context-free group. Brough has conjectured in \cite{Tara} that the class of finitely generated poly-context-free groups coincides with the class of (finitely generated) groups which are virtually subgroups of a direct product of free groups. Note that, in general, subgroups of direct products of (two or more) free groups can behave quite wildly. For instance, $\mathbb{F}_2\times \mathbb{F}_2 $ contains as many non-isomorphic subgroups \cite{Baum}, and for finitely generated subgroups they can even have unsolvable membership problem \cite{Miha}. 
 In this section, we prove a weaker version of Brough's conjecture by showing that the class of finitely generated groups which are virtually a finitely generated subgroup of a direct product of free groups coincides with the class of groups whose word problem is in the class of languages {\bf PT-ICF} that are the intersection of finitely many inverse-context-free languages accepted by inverse graphs that are quasi-transitive. By Theorem~\ref{theo: main theorem}, {\bf PT-ICF} is also the class of languages that is the intersection of finitely many languages accepted by quasi-transitive and tree-like inverse graphs.  
\begin{theorem}\label{main theorem poly-contex}
A group $G$ has word problem in the class {\bf PT-ICF}  if and only if it is virtually a (finitely generated) subgroup of a direct product of free groups. 
\end{theorem}
\begin{rem}
The class of groups having the word problem in {\bf PT-ICF} is included in the class of co-context-free groups introduced in \cite{holt}. This follows from Proposition~\ref{prop: equivalence context-free} since the language accepted by a context-free inverse graph, is, in particular, deterministic, and the fact that the union and complement of deterministic context-free languages is still context-free \cite[Theorem~10.1]{hop}. Thus, by the previous theorem and \cite{Miha}, we deduce that for the class of poly-context-free and co-context-free groups, the membership problem is undecidable.
\end{rem}
We recall that a \emph{subdirect} product of groups $H_1,\ldots, H_n$ is a subgroup $G\le H_1\times \ldots\times H_n$ which projects surjectively onto each factor. One direction of the proof of Theorem~\ref{main theorem poly-contex} is given by the following proposition. 
\begin{prop}
Let $G$ be a group having word problem $W=WP(G;A)=\bigcap_{i=1}^kL(\Gamma_i, x_i)$ where $(\Gamma_i, x_i)$ are rooted inverse graphs that are quasi-transitive and tree-like. Then, $G$ is virtually a subdirect product of $k$ free groups. 
\end{prop}
\begin{proof}
Let us recall that $\theta:\wt{A}^*\to \mathbb{F}_A$ denotes the natural projection $\theta(u)=\oo{u}$, and we put $W=WP(G;A)$. By hypothesis, $W=\bigcap_{i=1}^kL(\Gamma_i, x_i)$ where $(\Gamma_i, x_i)$ are rooted inverse graphs that are quasi-transitive and context-free. Note that $\theta(W)=\bigcap_{i=1}^k\oo{L(\Gamma_i, x_i)}$ is a normal subgroup of $\mathbb{F}_A$. In particular, for any $u\in\wt{A}^*$ we have $uu^{-1}\in WP(G;A)$, from which we conclude that each inverse graph $\Gamma_i$ is complete, i.e., for any vertex $v\in V(\Gamma_i)$ and $a\in \wt{A}$ there is an edge $v\mapright{a}v'$ in $\Gamma_i$. By Theorem~\ref{theo: main theorem} for each $i\in [1,k]$ there is a cover $\psi_i: \Gamma_i\to \Lambda_i$ with $\Lambda_i$ finite inverse graph and a surjective homomorphism $\oo{\eta}_i:\pi_1(\Lambda_i, \psi(x_i))\to \mathbb{F}_{X_i}$  onto a finite rank free group $\mathbb{F}_{X_i}$ with $\oo{L(\Gamma_i,x_i)}=\ker(\oo{\eta_i})$. Note that since each $\Gamma_i$ is complete, $\Lambda_i$ is also complete. Hence, $(\Lambda_i, \psi_i(x_i))$ is the Stallings automaton of $\pi_1(\Lambda_i, \psi_i(x_i))$ which is a finite index subgroup of $\mathbb{F}_A$, being $\Lambda_i$ complete, see for instance \cite[Chapter 23]{Handbook}. Therefore, $\mathcal{H}=\cap_i \pi_1(\Lambda_i, \psi_i(x_i))$ is a finite index subgroup of $\mathbb{F}_A$. Consider the product homomorphism $\varphi: \mathcal{H}\to \Pi_{i}^k \mathbb{F}_{X_i}$
defined by $\varphi(g)=(\oo{\eta}_1(g),\ldots, \oo{\eta}_k(g))$ for all $g\in\mathcal{H}$. Clearly $\ker\varphi=\bigcap_{i=1}^k\oo{L(\Gamma_i, x_i)}=\theta(W)$, hence $\mathcal{H}/\theta(W)$ is a finitely generated finite index subgroup of $G\simeq \mathbb{F}_A/\theta(W)$ that is isomorphic to a subgroup of $\Pi_{i}^k \mathbb{F}_{X_i}$ which projects surjectively onto each factor since each $\oo{\eta}_i$ is surjective. 
\end{proof}
To prove the other direction of Theorem~\ref{main theorem poly-contex} we need to recall the notion of inverse transducer. This notion has been firstly introduced in \cite{Pedro} as a tool to show that the fixed point subgroup of an endomorphism of a finitely generated virtually free group is finitely generated. A transducer which is also referred to as a generalized sequential machine is a finite $Y\times X^*$-digraph rooted at some vertex $p_0$. A generic edge may be graphically represented by $p_1\vlongmapright{y_1\mid u_1}p_2, y_1\in Y, u_1\in X^*$, while we may represent the walk
\[
p_1\vlongmapright{y_1\mid u_1}p_2\vlongmapright{y_2\mid u_2}p_3\ldots p_k\vlongmapright{y_k\mid u_k}p_{k+1}
\]
shortly as $p_1\vlongmapright{s\mid h}p_{k+1}$ where $s=y_1y_2\ldots y_k$, $h=u_1u_2\ldots, u_k$. A transducer $\mathcal{A}$ on the alphabet $\wt{Y}\times \wt{X}^*$ is called \emph{inverse}, if for any edge $p_1\vlongmapright{y\mid u}p_2$, $y\in \wt{Y}, u\in\wt{X}^*$, there is the reverse edge $p_2\vlongmapright{y^{-1}\mid u^{-1}}p_1$, and $\mathcal{A}$ is deterministic in the following sense: for any vertex $p$ and $z\in\wt{Y}$ there is at most one edge $p\vlongmapright{z\mid w}p'$ in $\mathcal{A}$. In this way $\mathcal{A}$ defines a partial function $\mathcal{A}:\wt{Y}^*\to \wt{X}^*$ by putting $\mathcal{A}(s)=h$ whenever in $\mathcal{A}$ there is a circuit $p_0\vlongmapright{s\mid h}p_0$. 
Given an inverse $\wt{X}$-graph $\Gamma$ and an inverse transducer $\mathcal{A}$ on the alphabet $\wt{Y}\times \wt{X}^*$ and root $p_0$, we may define the $\wt{Y}$-graph $\mathcal{A}\ltimes \Gamma$ having as set of vertices $V(\mathcal{A})\times V(\Gamma)$ and edges 
\[
(p_1, q_1)\mapright{y}(p_2,q_2)\mbox{ whenever }p_1\mapright{y\mid u}p_2\in \mathcal{A}, \,q_1\mapright{u}q_2\mbox{ is a walk in }\Gamma\mbox{ for some }u\in\wt{X}^*.
\]
It is straightforward to check that $\mathcal{A}\ltimes \Gamma$ is an involutive $\wt{Y}$-digraph that is deterministic. In general, it is missing the condition to be connected to make this graph inverse. Thus, we may think of $\mathcal{A}\ltimes \Gamma$ as a collection of disjoint union of its connected components each of which is an inverse graph. Note that by fixing a root $x_0$ in $\Gamma$ then the language $L(\mathcal{A}\ltimes \Gamma, (p_0,x_0))$ is the set of words $u\in \wt{Y}^*$ such that $p_0\vvlongmapright{u\mid \mathcal{A}(u)}p_0$ is a circuit in $\mathcal{A}$ with $\mathcal{A}(u)\in L(\Gamma, x_0)$. Thus, $L(\mathcal{A}\ltimes \Gamma, (p_0,x_0))=\mathcal{A}^{-1}(L(\Gamma, x_0))$ is the inverse image under the partial map defined by the transducer of the language $L(\Gamma, x_0)$. It is a well-known fact that context-free languages are closed under the inverse image of the partial map defined by a generic transducer, see for instance \cite[Theorem 11.2]{hop}, a fact that will be used later.
\begin{lemma}\label{lemma: product is quasi-transitive}
Let $\mathcal{A}$ be an inverse transducer on the alphabet $\wt{Y}\times \wt{X}^*$, and $(\Gamma,x_0)$ be a quasi-transitive inverse $\wt{X}$-graph. Then, $\mathcal{A}\ltimes \Gamma$ is a quasi-transitive inverse graph. In particular, if $L(\Gamma, x_0)$ is an inverse context-free language, then $\mathcal{A}^{-1}(L(\Gamma, x_0))\in {\bf PT-ICF}$. 
\end{lemma}
\begin{proof}
For every automorphism $\varphi\in\Aut(\Gamma)$ consider the map $\hat{\varphi}_v:V(\mathcal{A}\ltimes \Gamma)\rightarrow V(\mathcal{A}\ltimes \Gamma)$ defined by putting $\hat{\varphi}(p_1,q_1)=(p_1,\varphi(q_1))$. This map extends to a morphism of involutive graphs. Indeed, $(p_1,q_1)\mapright{y}(p_2,q_2)$ is an edge in $\mathcal{A}\ltimes \Gamma$ if and only if there is some $h\in\wt{X}^*$ such that $p_1\mapright{y\mid h}p_2$ is an edge in $\mathcal{A}$ and $q_1\mapright{h}q_2$ is a walk in $\Gamma$. Thus, by applying $\varphi$, in $\Gamma$ we have also the walk $\varphi(q_1)\mapright{h}\varphi(q_2)$, hence $(p_1,\varphi(q_1))\mapright{y}(p_2,\varphi(q_2))$ is an edge in $\mathcal{A}\ltimes \Gamma$. Thus, by putting 
\[
\hat{\varphi}_e((p_1,q_1)\mapright{y}(p_2,q_2))=(p_1,\varphi(q_1))\mapright{y}(p_2,\varphi(q_2))
\]
and by the determinism of $\mathcal{A}\ltimes \Gamma$, we have that $\hat{\varphi}:\mathcal{A}\ltimes \Gamma\to \mathcal{A}\ltimes \Gamma$ is a morphism of involutive $\wt{X}$-graphs. Furthermore, $\hat{\varphi}$ is an automorphism, since both $\hat{\varphi}_v$ and $\hat{\varphi}_e$ are bijective with inverses
\[
\hat{\varphi}_v^{-1}(p_1,q_1)=(p_1,\varphi^{-1}(q_1)),\quad \hat{\varphi}^{-1}_e((p_1,q_1)\mapright{y}(p_2,q_2))=(p_1,\varphi^{-1}(q_1))\mapright{y}(p_2,\varphi^{-1}(q_2))
\]
Thus, the map $\psi:\Aut(\Gamma)\to \Aut(\mathcal{A}\ltimes \Gamma)$ defined by $\psi(\varphi)= \hat{\varphi}$ is an homomorphism of groups which is also injective. Indeed, $\psi(\varphi_1)=\psi(\varphi_2)$ implies that in each connected component $\Theta$ of $\mathcal{A}\ltimes \Gamma$ and for any $(p_1,q_1)\in \Theta$ we have $\varphi_1(q_1)=\varphi_2(q_1)$, thus by Lemma~\ref{lem: uniqueness of automorphism} $\left.\varphi_1\right|_{\Theta}=\left.\varphi_2\right|_{\Theta}$, i.e., $\varphi_1=\varphi_2$. Since the action of $\Aut(\Gamma)$ on $\Gamma$ is quasi-transitive, the set $V(\Gamma)$ decomposes into finitely many orbits $O_1,\ldots, O_k$. Therefore, each $\Aut(\Gamma)$-orbit of $\mathcal{A}\ltimes \Gamma$ is contained in $ V(\mathcal{A})\times O_i$ for some $i\in\{1,\ldots, k\}$, thus $\Aut(\mathcal{A}\ltimes \Gamma)$ acts quasi-transitively on $\mathcal{A}\ltimes \Gamma$. The last statement follows from the equality $L(\mathcal{A}\ltimes \Gamma, (p_0,x_0))=\mathcal{A}^{-1}(L(\Gamma, x_0))$, which is a context-free language, being $L(\Gamma, x_0)$ context-free, which is also accepted by the quasi-transitive inverse graph $(\mathcal{A}\ltimes \Gamma, (p_0, x_0))$, i.e., $\mathcal{A}^{-1}(L(\Gamma, x_0))\in {\bf PT-ICF}$.
\end{proof}
\begin{prop}
Let $G$ be a group that is virtually a finitely generated subgroup of a direct product of free groups. Then, the word problem $W(G;Y)\in {\bf PT-ICF}$.
\end{prop}
\begin{proof}
Let $H$ be the finite index subgroup of $G$ that is isomorphic to a subgroup of the direct product $\Pi_{i=1}^k \mathbb{F}_{Z_i}$ of the free groups $\mathbb{F}_{Z_i}$. Let $X$ be a finite generating set of $H$, and let $T$ be a right transversal with $\mathds{1}_G\in T$. Every element of $g\in G$ may be represented as $g=_Ght$ for some $h\in H$, $t\in T$. Put $Y=X\cup\{T\setminus\{\mathds{1}_G\}\}$ as the generating set of $G$. For each $y\in\wt{Y}$ and $t\in T$, fix a word $h_{t,y}\in \wt{X}^*$ with $ty=_Gh_{t,y}t'$ for some $t'\in T$.  As in \cite{holt} we construct a transducer $\mathcal{A}$ in the following way. The set of vertices is $T$ and for for any $y\in Y$ and $t_1\in T$ we have the edge
\[
t_1\vvlongmapright{y\mid h_{t_1,y}}t_2\mbox{ whenever we have }t_1y=_Gh_{t_1,y}t_2
\]
and take $\mathds{1}_G$ as the root of $\mathcal{A}$. Note that $\mathcal{A}$ is inverse. Indeed, it is involutive since 
\[
t_1\vlongmapright{y\mid h_{t_1,y}}t_2\in\mathcal{A}\mbox{ if and only if }t_2\vvlongmapright{y^{-1}\mid h_{t_1,y}^{-1}}t_1\in\mathcal{A}.
\] 
and it is also deterministic since $y_{t_1,y}t_2=_Gy_{t_1,y}t'_2$ implies $t_2=t'_2$. Now, it is straightforward to check that for any $t\in T$ and $y\in \wt{Y}^*$, in $\mathcal{A}$ we have a walk $t\longmapright{y\mid u}t'$, for some $u\in \wt{X}^*$ if and only if $ty=_G ut'$. Let $\phi:\wt{Y}^*\to G$ denote the natural projections, and let $\pi_i:H\to \Pi_{i}^k \mathbb{F}_{Z_i}$ be the projection into the $i$-th component of the direct product. The group $H/\ker\pi_i$ is a finitely generated free group whose word problem $W_i\subseteq\wt{X}^*$ is a context-free language accepted by a quasi-transitive inverse graph $(\Gamma_i,x_i)$, namely the Cayley graph $\Cay(H/\ker\pi_i;X)$. Note that any word $w\in W(G;Y)$ belonging to the word problem is characterized by the conditions $\pi_i(\phi(w))=1$ for $i=1,\ldots, k$, that is, for any $i=1,\ldots, k$ there is some $h_i\in W_i=L(\Gamma_i, x_i)$ such that $w=_G h_i$, i.e., $\mathcal{A}(w)\in L(\Gamma_i, x_i)$. Hence, by Lemma~\ref{lemma: product is quasi-transitive} we conclude that the word problem
\[
W(G;A)=\bigcap_{i=1}^k \mathcal{A}^{-1}(L(\Gamma_i, x_i))\in {\bf PT-ICF}
\]
\end{proof}

%
%

\section*{Acknowledgments}
The author is grateful to W. Woess for the fruitful discussions regarding some equivalences contained in Theorem~\ref{theo: main theorem}.

\end{document}